\documentclass[leqno,12pt]{amsart}
\usepackage{xr}
\usepackage{hyperref}
\usepackage{graphicx,color}
\usepackage[usenames,dvipsnames,svgnames,table]{xcolor}
\usepackage{epstopdf}
\usepackage{amssymb,amsmath}
\usepackage{amsthm}
\usepackage{enumerate}
\newtheorem{theorem}{Theorem}[section]
\newtheorem{proposition}[theorem]{Proposition}
\newtheorem{lemma}[theorem]{Lemma}
\newtheorem{definition}[theorem]{Definition}

\newtheorem{remark}[theorem]{Remark}
\newtheorem{example}[theorem]{Example}
\numberwithin{equation}{section}

\newcommand{\C}{{\mathbb C}}
\newcommand{\R}{{\mathbb R}}
\newcommand{\Z}{{\mathbb Z}}
\newcommand{\N}{{\mathbb N}}
\newcommand{\Q}{{\mathbb Q}}
\newcommand{\e}{{\mathrm{e}}}

\newcommand{\CF}{{\mathcal F}}

\newcommand{\CH}{{\mathcal H}}

\newcommand{\CM}{{\mathcal M}}

\newcommand{\lattice}{\Lambda}

\newcommand{\CR}{{\mathcal R}}

\newcommand{\suppcone}{{C}}

\newcommand{\codim}{\operatorname{codim}}
\newcommand{\lin}{\operatorname{lin}}
\newcommand{\proj}{\operatorname{proj}}
\newcommand{\Res}{\operatorname{Res}}

\newcommand{\Sym}{\operatorname{Sym}}

\def\binomial(#1,#2){\binom{#1}{#2}}
\def\mult(#1,#2){\binom{#1}{#2}}
\renewcommand{\a}{{\mathfrak{a}}}

\renewcommand{\b}{{\mathfrak b}}
\renewcommand{\c}{{\mathfrak{c}}}
\renewcommand{\d}{{\mathfrak{d}}}
\newcommand{\f}{{\mathfrak{f}}}
\renewcommand{\t}{{\mathfrak{N}}}

\newcommand{\p}{{\mathfrak{p}}}
\newcommand{\q}{{\mathfrak{q}}}

\title[]{Local asymptotic Euler-Maclaurin expansion for Riemann sums over a semi-rational polyhedron}
\author{N. Berline}
\address{Nicole Berline:  \'Ecole Polytechnique, Centre de Math\'ematiques Laurent Schwartz, 91128 Palaiseau Cedex, France}
\email{nicole.berline@math.cnrs.fr}
\author{M. Vergne}
\address{Mich\`ele Vergne: Universit\'e Paris 7 Diderot, Imj-Prg, Sophie Germain,
75013 Paris} \email{michele.vergne@imj-prg.fr}

\begin{document}
\maketitle{}
\begin{abstract}
Consider the Riemann sum of a smooth compactly supported function $h(x)$ on a  polyhedron $\p\subseteq \R^d$, sampled at the points of the lattice $\Z^d/t$. We give an asymptotic expansion when $t\to +\infty$, writing each coefficient of this expansion as a sum indexed by the faces $\f$ of the polyhedron, where the $\f$ term is the integral over $\f$ of a differential operator  applied to the function $h(x)$. In particular, if a Euclidean scalar product is chosen, we prove that the differential operator for the face $\f$ can be chosen (in a unique way) to involve only normal derivatives to $\f$.
Our formulas are valid for  a semi-rational polyhedron and a real sampling parameter $t$, if we allow for \emph{step-polynomial} coefficients, instead of just constant ones.
\end{abstract}

\clearpage

\tableofcontents

\section{Introduction}

Let $\p\subset \R^d$ be a  convex polyhedron of dimension $\ell$. We assume that $\p$ is \emph{semi-rational} (its facets are  affine hyperplanes parallel to rational ones).  Let $h(x)$ be a smooth compactly supported function on $\R^d$.  In this article, we give an asymptotic expansion    of the Riemann sum
\begin{equation*}\label{eq:dilated-distribution}
  \langle R_t(\p),h\rangle = \frac{1}{t^\ell} \sum_{x\in t\p \cap \Z^d} h(\frac{x}{t}),
\end{equation*}
when $t\to +\infty$, $t\in \R$.

The basic example is given by the classical Euler-Maclaurin expansion on a half-line $[s,+\infty[$ with $s\in \R$.
For $x\in \R$, denote by $\{x\}\in [0,1[$ the fractional part of $x$.
 For $n>0$, and any real number $t>0$, we have
\begin{multline}\label{eq:intro-EML-dim1-asymptotic-vertex-s}
\frac{1}{t} \sum_{x\in \Z, x\geq ts}h(\frac{x}{t})=\\
 \int_s^\infty  h(x) dx -
 \sum_{k=1}^{n-1} \frac{1}{t^k} \frac{B_k(\{-t s\})}{k!}h^{(k-1)}(s)+ O( \frac{1}{t^n}).
 \end{multline}

This formula can be considered as an asymptotic expansion when $t\to +\infty$, $t\in \R$, (with  a closed expression for the remainder, see Formula (\ref{eq:EML-dim1-asymptotic-vertex-s})), where we allow the coefficient of $\frac{1}{t^k}$ to be  a so-called \emph{step-polynomial function} of $t$, (here, the Bernoulli polynomial computed at the fractional part $\{-t s\}$). If the end-point  $s$ is an integer and if the  parameter $t$ is also restricted to integers, (\ref{eq:intro-EML-dim1-asymptotic-vertex-s}) becomes the more familiar asymptotic expansion with constant coefficients
\begin{equation}\label{eq:intro-EML-dim1-asymptotic-vertex-integer}
\frac{1}{t} \sum_{x\in \Z, x\geq ts}h(\frac{x}{t})=
 \int_s^\infty  h(x) dx -
 \sum_{k=1}^{n-1} \frac{1}{t^k} \frac{b_k}{k!}h^{(k-1)}(s)+ O( \frac{1}{t^n})
 \end{equation}
 where $b_k=B_k(0)$ are the Bernoulli numbers.

 \bigskip

The asymptotic expansion which we obtain in this paper is a generalization of (\ref{eq:intro-EML-dim1-asymptotic-vertex-s}) to any {semi-rational} polyhedron $\p$. We prove that $R_t(\p)$ has an asymptotic expansion
$$
R_t(\p)\sim \int_\p h(x) + \sum_{k\geq 1} a_k(t)\frac{1}{t^k}
$$
where $a_k(t)$ is a step-polynomial function of $t$ (see Definition \ref{def:step-poly}), and, given a Euclidean scalar product,  we write explicitly   $a_k(t)$ as a sum of integrals over all the (proper) faces of $\p$,
\begin{equation}\label{eq:intro-normal-derivatives}
 a_k(t)= \sum_\f \int_{\f} D_{k,\f,t}\cdot h,
\end{equation}
  where $ D_{k,\f,t}$ is a differential operator of degree $k+ \dim\f-\dim \p$ involving only derivatives normal to the face $\f$, with coefficients which are step-polynomials  functions of $t$. With these   conditions, the operators $D_{k,\f,t}$ are unique.

If $\p$ is rational, then the coefficients of the operators  $D_{k,\f,t}$ are periodic functions of $t$ of period $q$, where $q$ is an integer, and the coefficients $a_k(t)$ of the asymptotic expansion are periodic functions of $t$.
If moreover $\p$ is a lattice polyhedron (every face contains an integral point), the period $q$ is $1$, so
if, in addition, the parameter $t$ is restricted to integral values,  then the coefficients $a_k(t)$ are just constants.

If $\p$ is a polytope (compact polyhedron) and if $h(x)$ is a polynomial, the asymptotic formula is an exact finite expansion and coincides with our previous \emph{local Euler-Maclaurin expansion} for polynomial functions \cite{MR2343137}.
However, we do not use this previous result, and we give an elementary  proof directly in the $C^{\infty}$-context.

Let us explain our approach.

If $\p$ is compact, the  Fourier transform of the distribution $ R_t(\p)$ is given by the holomorphic function of $\xi$
\begin{eqnarray}\label{eq:intro-fourier}
\CF(R_t(\p))(\xi)= \frac{1}{t^\ell}\sum_{x\in t\p \cap \Z^d} \e^{-i \langle\frac{\xi}{t}, x\rangle}.
\end{eqnarray}
If $\p$ is not compact, but pointed (that is $\p$ does not contain any affine space),  the  function
$$
S(\p)(\xi)=\sum_{x\in \p \cap \Z^d} \e^{\langle\xi, x\rangle}
$$
can still be defined as a meromorphic function  with simple real hyperplane singularities around $\xi=0$. The Fourier transform
$\CF(R_t(\p))(\xi)$
is  a generalized function equal to a boundary value of the meromorphic function $ \frac{1}{t^\ell}S(t\p)(-i\frac{\xi}{t})$. The archetype is $\p=[0,\infty[$ where
\begin{equation}
\CF(R_t(\p))(\xi)=\frac{1}{t}\lim_{\epsilon\to 0, \epsilon>0} \frac{1}{1-\e^{-i\frac{\xi}{t}-\epsilon}}.
\end{equation}
When $\p$ is a pointed affine cone, we obtain a canonical  asymptotic expansion of the Fourier transform $\CF(R_t(\p))(\xi)$ in terms of the Laurent series of $S(t\p)(\xi)$ at $\xi=0$, (Theorem \ref{th:asymptotic-cone} in the case of cone with vertex $0$, Theorem \ref{th:asymptotic-cone-real} for an affine cone with any real vertex). This result is easily reduced to the one-dimensional Euler-Maclaurin formula, by subdividing a cone into unimodular ones. Nevertheless, it is the most important observation of this paper.

If the support of the test function is small enough, the Riemann sum for a polyhedron is actually a Riemann sum on one of its supporting cones. Such a  cone is  a product of a linear space and of a pointed affine cone.
So, by a partition of unity argument, the basic result we need is for a pointed affine cone $\c$.
In this case,  by Fourier transform,
 Equation (\ref{eq:intro-normal-derivatives}) is equivalent  to a decomposition (depending of a  choice of a scalar product) of
$S(\c)$
 into a sum of terms indexed by the faces $\f$ of $\c$, where  each term  is  a product of a term involving $\f$ and a term involving the transverse cone to $\c$ along $\f$.
 This expression for $S(\c)$  (Theorem \ref{th:local-EML-cone})  was already obtained in \cite{MR2343137}, but we give here a simpler  proof in the appendix. The easier case of simplicial cones is explained in Subsection \ref{subsec:localsimplicial}.

\bigskip

 Our motivation was to clarify and simplify some previous results on such asymptotic expansions, \cite{MR2184566}, \cite{MR2737411},  \cite{LeflochPelayo2013}. In \cite{MR2737411}, T. Tate obtained an asymptotic expansion with normal derivatives as in (\ref{eq:intro-normal-derivatives}) for a lattice polytope and an integral parameter $t$. In \cite{LeflochPelayo2013}, Y. Le Floch and A. Pelayo showed  in small dimensions how to derive Tate's result from the Todd operator asymptotic formula of V. Guillemin and S. Sternberg \cite{MR2184566}.

 Along the way, we realized that our method permits nice generalizations.

First, our formula is valid for any polyhedron, we do not need to assume that it is compact, nor even pointed.
Another generalization is to consider a real parameter $t$ instead of an integer for the Riemann sum. Even if we start with a rational polyhedron $\p$, the dilated polyhedron $t\p$ is only semi-rational, so we assume only that $\p$ is semi-rational.

\medskip

The proofs  in the case of a semi-rational polyhedron and a real parameter $t$
are very similar to  those for  a lattice polyhedron and an integral parameter, so we give the details only in that case.

\bigskip

 The differential operators in our formula are those which we constructed in \cite{MR2343137}. There,
 for  a face $\f$ of the polyhedron $\p$, the  symbol of operator corresponding to the face $\f$ is the $\mu$-function of the transverse cone to $\p$ along the face $\f$.  S.  Paycha observed that this $\mu$-function can be  defined by a more intuitive method of {algebraic renormalization} \cite{GuoPaychaZhang2015}.   Although this renormalization construction is  easy for the generating function of a simplicial cone, we explain it, in the Appendix,  in the more general context of rational functions with   hyperplane singularities, on any base field. For computing the $\mu$-function of a cone, we wrote a Maple program (\cite{RenormalizationMaple}, with V. Baldoni).

\bigskip

We had to leave open the problem of writing a closed formula for the remainder at order $n$ of the asymptotic expansion,  similar to the one dimensional remainder in (\ref{eq:EML-dim1-asymptotic-vertex-s}), in the spirit of the present article, (a remainder for a simple polytope, in the spirit of the Todd operator formula, is obtained in \cite{MR2184566}). We only give a small computation in Example \ref{ex:remainder}.

We thank the referee for his careful reading.

\section{Notations and basic facts}
\subsection{Various notations}
\subsubsection{}
\noindent $V$ is a finite dimensional vector space over $\R$ with a lattice $\lattice$. The dimension of $V$ is denoted by $d$. The Lebesgue measure $dx$ on $V$ is determined by $\lattice$.

 Elements of $V$ are denoted by Latin letters $x,y ,v \ldots$ and elements of the  dual space $V^*$ are denoted by Greek letters $\xi,\gamma \ldots$. The pairing between $V$ and $V^*$ is denoted by $\langle\xi,x\rangle$.

\subsubsection{Bernoulli numbers and polynomials}
$$
\frac{z \e^{sz}}{\e^z-1} = \sum_{n=0}^\infty B_n(s)\frac{z^n}{n!}, \;\;\;\frac{z }{\e^z-1} = \sum_{n=0}^\infty b_n \frac{z^n}{n!}.
 $$
\begin{eqnarray*}
B_0(s) &=&  1, B_1(s)= s-\frac{1}{2}, B_2(s)= s^2-s+\frac{1}{6}, B_3(s)= s^3-\frac{3}{2}s^2 +\frac{1}{2}s.\\
 b_0&=& 1, \;\;\; b_1 = -\frac{1}{2}, \;\; \;b_2 = \frac{1}{6}, \;\; \; b_{2n+1}= 0 \mbox{ for } n\geq 1.
 \end{eqnarray*}
\subsubsection{Fractional part of a real number}
For $s\in \R$, the fractional part $\{s\}$ is defined by $\{s\}\in [0,1[, s-\{s\}\in \Z$.
\subsubsection{Fourier transform}
The  Fourier transform of a distribution on $V$ is a generalized function on $V^*$. Our convention for the Fourier transform of a test density $\phi(x) dx $  is
$$
\CF(\phi dx)(\xi)=\int_{V}\e^{-i\langle \xi,x\rangle}\phi(x)dx.
$$
\subsection{Polyhedra, cones.}\label{subsection:generating-function}
\subsubsection{}
A \emph{convex  polyhedron} $\p$ in $V$ (we will simply say
 \emph{polyhedron}) is, by definition, the intersection of a finite number of
closed half spaces bounded by   affine hyperplanes. If the
hyperplanes are rational, we say that the polyhedron is \emph{rational}. If
the hyperplanes have rational directions, we say that the polyhedron
is \textit{semi-rational}.
 For instance, if $\p\subset V $ is a
rational polyhedron, $t$~is a real number and $s$~is any point in~$V$, then
the dilated polyhedron $t\p$ and the translated polyhedron
$s+\p$ are semi-rational.

The \emph{lineality space} of a polyhedron $\p$ is the subspace of $y\in V$ such that $x+\R y\subseteq  \p$ for all $x\in \p$. A polyhedron is called pointed when its lineality space is $\{0\}$.

In this article, a cone is a rational polyhedral cone (with vertex $0$) and
an affine cone is a translated set $s+\c$ of a cone $\c$ by any element $s\in V$. A  cone
 is called \emph{simplicial} if it is  generated by independent
elements of $\Lambda $. A  cone  is called \emph{unimodular} if it
is generated by independent elements $v_1,\dots, v_k$ of $\Lambda$ such
that $\{v_1,\dots, v_k\}$ can be completed to an integral basis of
$\Lambda$. An affine cone $\a$ is called \emph{simplicial} (resp.\ \emph{simplicial
unimodular}) if the associated cone is.

For a polyhedron $\f$, the affine span of $\f$ is denoted by $\langle \f\rangle$ and the corresponding linear space is denoted by $\lin ( \f) $.

If $\f$ is a face of a polyhedron $\p$, the lineality space of the supporting cone of $\p$ along $\f$ is just $\lin(\f)$.  The projected cone  in the quotient space $V/\lin(\f)$ is a pointed cone called the \emph{transverse cone} of $\p$ along $\f$.
\begin{definition}\label{def:transverse-cone}
The supporting cone of $\p$ along $\f$ is denoted by $C(\p,\f)$.

The transverse cone of $\p$ along $\f$ is denoted by  $\t(\p,\f)$.
\end{definition}

The Lebesgue measure on $\langle \f\rangle$ is determined by the intersection lattice $\lin(\f)\cap \lattice$. It is denoted by $dm_\f$.

\subsubsection{Subdivision of a cone into unimodular ones}\label{subsection:cone-decomposition}
A subdivision of a cone $\c$ is a family of cones $\c_\alpha$ contained in $\c$, such that the intersection of any two of them is a face of both and belongs to the family and such that $\c$ is the union of the family. Then the indicator  $[\c]$ is equal to a linear combination  $[\c]=\sum_{\alpha}n_\alpha [\c_\alpha]$ with $n_\alpha\in \Z$.

It follows from Minkowski's theorem that any cone has a subdivision into unimodular cones.

\subsection{Discrete and continuous generating functions of a pointed polyhedron $S(\p), I(\p)$}
If $\p\subset V$ is a pointed polyhedron, there is an $a>0$ and a non empty open set $U\subseteq V^*$ such that for $\xi\in U$  and  $x\in \p$ large enough,  $\langle\xi,x\rangle \leq -a \|x\|$. Then the functions
$$
S(\p)(\xi)=\sum_{x\in \p\cap \lattice}\e^{\langle \xi,x\rangle}, \;\;\; I(\p)(\xi)=\int_\p \e^{\langle \xi,x\rangle}dm_\p(x)
$$
are defined and holomorphic on $U+iV^*\subseteq V^*_\C$, and have meromorphic continuation to the whole of $V^*_\C$. If $\p$ is bounded, these functions are holomorphic.

If $\c$ is a simplicial cone with  edge generators $v_1,\ldots,v_\ell\in \Lambda$, and $s\in V$, one has
$$
I(s+\c)(\xi)= (-1)^{\ell}\e^{\langle \xi,s\rangle}\frac{|\det_\lattice(v_j)|}{\prod_j \langle \xi,v_j\rangle},
$$
where the determinant is relative to the intersection lattice $\lin(\c)\cap \lattice$, and
$$
S(s+\c)(\xi)=S(s+\b)(\xi)\prod_j \frac{1}{1-\e^{\langle \xi,v_j\rangle}}
$$
 where $S(s+\b)(\xi)$ is the holomorphic function
  $$
  S(s+\b)(\xi)=\sum_{x\in (s+\sum_j[0,1[v_j)\cap \lattice}\e^{\langle \xi,x\rangle}.
  $$
 If moreover $\c$ is unimodular and the $v_j$'s are primitive vectors, then, for the vertex $s=0$,
 $$
S(\c)(\xi)= \frac{1}{\prod_j(1-\e^{\langle \xi,v_j\rangle})}, \; \; \;  I(\c)(\xi)=(-1)^{\ell}\frac{1}{\prod_j \langle \xi,v_j\rangle}.
 $$
 If $\c$ is simplicial, the function
 $g(\xi)= (\prod_{j =1}^\ell \langle \xi,v_j\rangle) S(s+\c)(\xi)$ is holomorphic near $\xi=0$, in other words, $S(s+\c)(\xi)$ has \emph{simple hyperplane singularities} defined by its edges, near $\xi=0$. If $\xi$ is not in any of these singular hyperplanes, the function
 $ z\mapsto S(s+\c)(z\xi)$ is meromorphic,  with  Laurent series around $z=0$
 $$
 S(s+\c)(z \xi)=\sum_{m=-\ell}^\infty  S(s+\c)_{[m]}(\xi)z^m,
 $$
 where $S(s+\c)_{[m]}(\xi)= \frac{g_{[m+\ell]}(\xi)}{\prod_{j=1}^\ell \langle\xi, v_j\rangle}$ is a homogeneous rational fraction of degree $m$ which we call the homogeneous component of degree $m$ of $S(s+\c)(\xi)$. We write formally
 \begin{equation}\label{eq:Ssc}
  S(s+\c)( \xi)=\sum_{m=-\ell}^\infty  S(s+\c)_{[m]}(\xi).
 \end{equation}

If $\c$ is no longer assumed simplicial,  then $\c$ can be subdivided into simplicial cones without adding edges. Therefore, again, $S(s+\c)(\xi)$ has simple hyperplane singularities near $\xi=0$ (defined by the edges) and the decomposition into homogeneous components (\ref{eq:Ssc}) still holds.

\section{Asymptotic expansions for Riemann sums over a  cone }\label{section:cone-vertex-0}

\subsection{Dimension one}
 We recall  the dimension one Euler-Maclaurin summation formula for a test function $h(x)$ on a half line, (\cite{MR2312338}, Theorem 9.2.2). For the moment, we consider only the half-line $[0,+\infty[$.
 \begin{multline}\label{eq:EML-dim1-usual}
  \sum_{x\in \Z, x\geq 0}h(x)= \\
  \int_0^\infty h(x)dx -\sum_{k=1}^n  \frac{b_k}{k!}h^{(k-1)}(0)-  \int_0^\infty   \frac{B_n(\{- x\})}{n!}h^{(n)}(x)  \, dx.
 \end{multline}
 We write the whole right hand side as an integral over the half line.
\begin{multline}\label{eq:EML-dim1}
    \sum_{x\in \Z, x\geq 0}h(x)=
    \int_0^\infty \left(h(x)+\sum_{k=1}^n  \frac{b_k}{k!}h^{(k)}(x)-  \frac{B_n(\{- x\})}{n!}h^{(n)}(x) \right) \, dx.
\end{multline}
For $t>0$, consider the scaled function
$$
h_t(x)=\frac{1}{t}h(\frac{x}{t}).
$$
Substituting $h_t$ for $h$  in this formula, and changing variables in the integral on the right-hand-side, we obtain, for any $t>0$,
\begin{multline}\label{eq:EML-dim1-asymptotic-vertex0}
\frac{1}{t} \sum_{x\in \Z, x\geq 0}h(\frac{x}{t})=\\
 \int_0^\infty  \left( h(x)  +
 \sum_{k=1}^{n-1} \frac{1}{t^k} \frac{b_k}{k!}h^{(k)}(x)+
  \frac{1}{t^n}  \frac{b_n-B_n(\{-t x\})}{n!} h^{(n)}(x)\right)dx.
\end{multline}
(\ref{eq:EML-dim1-asymptotic-vertex0}) gives the asymptotic expansion when $t\to +\infty$,
$$
\frac{1}{t} \sum_{x\in \Z, x\geq 0}h(\frac{x}{t})\sim \sum_{k=0}^{\infty} \frac{1}{t^k} \int_0^\infty \frac{b_k}{k!}h^{(k)}(x),
$$
with a closed formula for the remainder.
\subsection{Fourier transforms and boundary values}
 For a polyhedron $\p\subset V$, consider the distributions  given on a test function $h$ on $V$ by
$ \int_\p h(x)dx $ and $\sum_{x\in \p\cap \lattice}h(x)$.
Clearly, they are  tempered distributions, therefore we can consider their Fourier transforms.

We recall a well known result on Fourier transforms.
\begin{definition}\label{def:limlambda}
If $F(\xi)$ is a rational function on $V^*$  whose denominator $g(\xi)$ is a product of linear forms  and if   $\lambda\in V^*$ is  such that $g(\lambda)\neq 0$, the following formula defines a tempered generalized function on $V^*$:
    $$
    \lim\limits_\lambda( F(\xi)) =\lim_{\epsilon\to 0, \epsilon > 0 }F(\xi+ i \epsilon \lambda).
    $$
\end{definition}

\begin{example}
    Let $v$ be a non zero vector in $V$ and let $\lambda$ be such that $\langle\lambda,v\rangle < 0$. Then,
 $\lim\limits_\lambda (\frac{1}{i \langle \xi,v\rangle})$  is the Fourier transform of the Heaviside distribution of the half-line $\R_{\geq 0} v$, given by $\int_0^\infty h(tv)dt$.

 When $V=\R$, the Fourier transform of the Heaviside distribution of the half-line $\R_{\geq 0}$ is more often denoted by  $\frac{1}{i(\xi-i0)}$.
\end{example}

Of course, such  boundary values are defined for more general types of meromorphic functions on $V^*_\C$. In this paper we will need only the rational functions of Definition \ref{def:limlambda} and the generating functions of cones. The following result is immediate.
\begin{proposition}\label{prop:fourier-transform}
    Let $\c\subset V$ be a pointed cone (with vertex $0$). Let $\lambda\in V^*$ be any element such that $-\lambda$ lies in the dual cone of $\c$. Then,

 \noindent (i) the boundary value
 $$    \lim\limits_\lambda(I(\c)(-i\xi))= \lim_{\epsilon\to 0, \epsilon>0 }I(\c)(-i\xi+\epsilon \lambda)
 $$
    exists and it is  the Fourier transform of the distribution $\int_\c h(x)d x$,

\noindent (ii) the boundary value
    $$
    \lim\limits_\lambda(S(\c)(-i\xi))= \lim_{\epsilon\to 0, \epsilon>0 }S(\c)(-i\xi+\epsilon \lambda)
    $$
    exists and it is  the Fourier transform of the distribution $\sum_{x\in \c \cap \lattice} h(x)$.
    In particular, these boundary values do not depend on the choice of $-\lambda$ in the dual cone.
\end{proposition}
\begin{example}
   For $\c= \R_{\geq 0}$, the  Fourier transform of the distribution    $ h\mapsto \sum_{n=0}^\infty h(n)$ is the boundary value
    $    \lim\limits_{\epsilon\to 0, \epsilon>0 }\frac{1}{1-\e^{-i\xi-\epsilon }}$.
\end{example}

\subsection{Asymptotic expansion of the Fourier transform of a Riemann sum over a cone}
For a polyhedron $\p\subseteq V$,
with $\dim\p=\ell$,   we consider the  Riemann sum
\begin{equation}\label{eq:dilated-distribution}
  \langle R_t(\p),h\rangle = \frac{1}{t^\ell} \sum_{x\in t\p \cap \lattice} h(\frac{x}{t}).
\end{equation}

If $\c$ is a pointed cone of dimension $\ell$, it follows immediately from Proposition \ref{prop:fourier-transform} that the Fourier transform of $R_t(\c)$ is given by
\begin{equation}\label{eq:dilated-fourier}
    \CF(R_t(\c))(\xi)=  \frac{1}{t^\ell} \lim\limits_\lambda (S(\c)(-i \frac{\xi}{t})).
\end{equation}
When $t\to \infty$, $\frac{\xi}{t}\to 0$.   Replacing  $\xi$ by $-i\frac{\xi}{t}$  in  the formal expansion  into homogeneous components   $  S(\c)(\xi)=\sum_{k=-\ell}^\infty S(\c)_{[k]}(\xi) $, we obtain formally \begin{equation}\label{eq:dilated-formal-homogeneous}
\frac{1}{t^\ell} S(\c)(-i \frac{\xi}{t})=\sum_{k=0}^\infty \frac{1}{t^k} S(\c)_{[k-\ell]}(-i\xi).
\end{equation}
It is a remarkable fact that this formal expansion leads to an  asymptotic expansion of boundary values. We are going to derive it  from the dimension one Euler-Maclaurin formula.
\begin{theorem}\label{th:asymptotic-cone}
Let $\c\subset V $ be a  pointed cone of dimension $\ell$. Let $\lambda\in V^*$ be such that $ -\lambda$ lies in the dual cone of  $\c$.
 Consider the  distribution on $V$ given by the Riemann sum
$$
 \langle R_t(\c) ,h\rangle =\frac{1}{t^\ell}\sum_{x\in \c \cap \lattice}h(\frac{x}{t}).
$$
It  has an asymptotic expansion when $t\to \infty$, the Fourier transform of which is given by
\begin{equation}\label{eq:EMLasymptotic-anycone}
 \CF( R_t(\c) )=t^{-\ell} \lim\limits_\lambda (S(\c)(-i \frac{\xi}{t})) \sim  \sum_{k=0}^\infty \frac{1}{t^k}\lim\limits_\lambda S(\c)_{[k-\ell]}(-i\xi).
\end{equation}
\end{theorem}
\begin{proof} Let us first look at the case where $\c=\R_{\geq 0}$. By writing  the Fourier transform of the asymptotic expansion (\ref{eq:EML-dim1-asymptotic-vertex0}), we obtain immediately (\ref{eq:dim1-fourier}) below. (This is the reason why we rewrote  Euler-Maclaurin formula (\ref{eq:EML-dim1}) as an integral over the half-line).

\begin{eqnarray}\label{eq:dim1-fourier}
\CF( R_t(\R_{\geq 0}) ) &\sim & \sum_{k=0}^\infty \frac{1}{t^k}\frac{b_k}{k!}
 \frac{(-i\xi)^k }{i(\xi-i0)}.
\end{eqnarray}

On the other hand,   we write the well known Laurent series
$$
S(\R_{\geq 0})(\xi )= \frac{1}{1-\e^{\xi}}= (\frac{-1}{\xi}) \sum_{k=0}^\infty \frac{b_k}{k!}\xi^k,
$$
hence
$$
 S(\R_{\geq 0})_{[k-1]}(-i\xi-\epsilon)=  \frac{b_k}{k!}\frac{(-i\xi-\epsilon)^k}{i\xi+\epsilon}.
$$
This  proves the theorem in the case where $\c=\R_{\geq 0}$, and so when $\c$ has dimension one. For the general case, we  consider a subdivision   of $\c$ into unimodular cones $\c_\alpha$, so that $[\c]=\sum_\alpha n_\alpha [\c_\alpha]$. The cones $\c_\alpha$ are contained in $\c$, so that $-\lambda$ belongs to the dual cone of $\c_\alpha$ as well. Let $\ell_\alpha$ be the dimension of $\c_\alpha$.
 For a test function $h$ and $t>0$, we have
\begin{multline*}
\langle R_t(\c),h\rangle = \frac{1}{t^\ell} \sum_{x\in \c\cap \lattice}h(\frac{x}{t})
=\sum_\alpha n_\alpha \frac{1}{t^\ell} \sum_{x\in \c_\alpha \cap \lattice}h(\frac{x}{t})\\
=\sum_\alpha n_\alpha \frac{1}{t^{\ell-\ell_\alpha}} \langle R_t(\c_\alpha) ,h\rangle.
\end{multline*}
Thus we may assume that $\c$ is unimodular. In this case the theorem follows immediately from the dimension one case.
\end{proof}

\subsection{Asymptotic expansions of Riemann sums over cones in terms of differential operators}

In this section, we  explain how one obtains  formulas of Euler-Maclaurin type for the asymptotic expansion of the Riemann sum itself, by taking inverse Fourier transforms in  Theorem \ref{th:asymptotic-cone}.

By Theorem  \ref{th:asymptotic-cone}, the coefficients of the asymptotic expansion of the Riemann sum are given by
\begin{equation}\label{eq:EMLasymptotic-anycone-Fourier-inverse}
\langle R_t(\c),h\rangle  \sim \sum_{k=0}^\infty \frac{1}{t^k}\langle F_k,h\rangle,  \mbox{ with }
F_k = \CF^{-1}(\lim\limits_\lambda (S(\c)_{[k-\ell]}(-i\xi))).
\end{equation}

For $k=0$, we have $S(\c)_{[-\ell]}= I(\c)$, and, (Proposition \ref{prop:fourier-transform},(i)),
$$
\langle\CF^{-1}(\lim\limits_\lambda I(\c)(-i\xi)), h\rangle= \int_\c h(x)dm_\c(x).
$$
So we recover the fact that  $ \langle F_0,h\rangle$ is the integral of $h$ over $\c$ (with respect to the Lebesgue measure $dm_\c(x)$ defined by the intersection lattice $\lin(\c)\cap \lattice$).

For $k=1$, it is well known that
$$
\langle F_1,h\rangle =\frac{1}{2}\int_{\partial \c}h
$$
where ${\partial \c}$ denotes the boundary of $\c$, union of the facets of $\c$, and the Lebesgue measure on each facet $\f$ is again  defined by the intersection lattice $\lin(\f)\cap \lattice$. This formula is true for any lattice polyhedron. For a unimodular  cone,  it follows easily from (\ref{eq:EMLasymptotic-cone}). If the cone is not unimodular, we do a  subdivision, as in the proof of Theorem \ref{th:asymptotic-cone}.

Let $(v_1,\ldots, v_n)$ be  the generators of the edges of $\c$. The homogeneous component $S(\c)_{[k-\ell]}(\xi)$
can be written (in many ways) as a sum
\begin{equation}\label{eq:differential-operators}
S(\c)_{[k-\ell]}(\xi)=\sum_J \frac{P_J(\xi)}{\prod_{j\in J}\langle \xi,v_j\rangle}
\end{equation}
  where $J\subseteq (1,\ldots, n)$ is such that  $(v_j,{j\in J})$ are linearly independent, and $P_J$ is a homogeneous polynomial of degree $k-\ell+|J|$. Taking inverse Fourier transform of boundary values,  using again Proposition \ref{prop:fourier-transform}(i),   we obtain
$$
 \langle F_k,h\rangle
     = \sum_J a_J \int_{\c_J}P_J(\frac{\partial}{\partial x})\cdot h(x) dm_{\c_J}(x),
$$
where $\c_J$ is the cone generated by $(v_j,j\in J)$ and $a_J=\frac{1}{|\det(v_j,j\in J)|}$ (determinant with respect to the intersection lattice on the subspace spanned by $(v_j,j\in J)$).

\begin{example}
Let  $V=\R^2$ with lattice $\Z^2$ , standard basis $(e_1,e_2)$  and coordinates $x_1,x_2$.
Let $\c$ be the cone generated by  $v_1=e_1, v_2=e_1+e_2$, with
    edges $\f_1=\R_{\geq 0} v_1$,  $\f_2=\R_{\geq 0} v_2$. The homogeneous components of $S(\c)$ can be computed by taking the product of the Todd series
 $$
\frac{1}{1-\e^{v_j}}= -\frac{1}{v_j}(1+\sum_{k\geq 1}\frac{b_k}{k!}v_j^k).
$$
$$
S(\c)(\xi_1,\xi_2) = \frac{1}{(1-\e^{\xi_1})(1-\e^{\xi_1+\xi_2})}
= \frac{1}{\xi_1(\xi_1+\xi_2)}-\frac{1}{2\xi_1}-\frac{1}{2(\xi_1+\xi_2)}
$$
$$
+\frac{1}{4}+\frac{1}{12}\frac{\xi_1+\xi_2}{\xi_1}+ \frac{1}{12}\frac{\xi_1}{\xi_1+\xi_2}+\cdots.
$$
We write the homogeneous component of degree $0$ in two different ways.
\begin{eqnarray}
\label{eq:A}
S(\c)_{[0]}(\xi_1,\xi_2)&= &\frac{1}{4}+\frac{1}{12}\frac{\xi_1+\xi_2}{\xi_1}+ \frac{1}{12}\frac{\xi_1}{\xi_1+\xi_2}\\
\label{eq:B}
&=&\frac{3}{8}+ \frac{1}{12}\,{\frac {\xi_2}{\xi_1}}+\frac{1}{24}\,{\frac {\xi_1-\xi_2}{\xi_1+\xi_2}}.
\end{eqnarray}
The corresponding expressions for the term $F_2$ in the asymptotic expansion (\ref{eq:EMLasymptotic-anycone-Fourier-inverse}) of $R_t(\c)$ are
\begin{eqnarray}
\label{eq:AA}
\langle F_2,h\rangle &=&\frac{1}{4} h(0,0) -\frac{1}{12}\int_{\f_1}\partial_{v_2} h
-\frac{1}{12}\int_{\f_2}\partial_{v_1} h \\
\label{eq:BB}
   &=& \frac{3}{8} h(0,0)-\frac{1}{12}\int_{\f_1}\partial_{x_2} h
-\frac{1}{24}\int_{\f_2}(\partial_{x_1}-\partial_{x_2})h.
\end{eqnarray}
\end{example}

\section{Simplicial cones and Normal derivatives formula}

\subsection{Asymptotic expansions of Riemann sums and Todd operator}
In this subsection, we compare  Theorem \ref{th:asymptotic-cone}
with  Guillemin-Sternberg  formula for a unimodular cone. This comparison  will not be used in this article.

Let $\c$ be a unimodular cone with primitive edge generators $v_j$.
One can compute the distributions $F_k$   by expanding $S(\c)=\prod_j \frac{1}{1-\e^{v_j}}$ as a product of Todd series.  It is easy to see that the resulting expansion can be written in a unique way in the form
$$
\sum_{m\geq 0, J\subseteq (1,\ldots, \ell ) }\frac{G_{m,J}}{\prod_{j\in J}v_j},
$$
where $G_{m,J}(\xi)$ is a homogeneous polynomial of degree $m$ belonging to  the algebra generated by $v_k, k\notin J$.
Our theorem \ref{th:asymptotic-cone} gives thus the following explicit formula:

\begin{equation}\label{eq:a-la-LeFloch-Pelayo}
\frac{1}{t^\ell}\sum_{x\in \c \cap \lattice} h(\frac{x}{t})\sim
\sum_{k=0}^\infty  \frac{1}{t^k}
\sum_{\stackrel{m\geq 0, J}{ m+\ell-|J|=k} }\int_{\f_J} G_{m,J}(\frac{\partial}{\partial_x})\cdot h(x)\, dm_{\f_J}(x).
\end{equation}
Here $\f_J$ is the face generated by the edges $(v_j, j\in J)$, the corresponding constant coefficient  differential operator
$G_{m,J}(\partial)$ is  homogeneous of degree $m$ and belongs  to the algebra generated by $(\partial_{ v_k}, k\notin J)$.

Le Floch-Pelayo \cite{LeflochPelayo2013} observed that (\ref{eq:a-la-LeFloch-Pelayo}) follows from
Guillemin-Sternberg formula for the asymptotic expansion  on a Riemann sum over  the unimodular cone $\c$,  in terms of a Todd operator
(Theorem 3.2 in  \cite{MR2394539}).
Let us recall this formula.
The  unimodular cone $\c$ with primitive edge generators $v_j$ is defined by the inequalities $\langle\lambda_j, x\rangle\leq 0$,
where
$\lambda_j$ is the basis of the lattice $\Lambda^*$ dual to $v_j$. Let
 $\c(a)$ be the affine  cone defined by the inequalities $\langle\lambda_j, x\rangle\leq a_j$.
Let  $$Todd(z)=\frac{z}{1-e^{-z}}=\sum_{k=0}^\infty (-1)^k \frac{b_k}{k!} z^k.$$
Consider the formal power series (in powers of $\frac{1}{t}$) of constant coefficients differential operators:
$$
Todd \biggl(\frac{1}{t}\frac{\partial}{\partial_a}\biggr)=\prod_{j=1}^d Todd\biggl(\frac{1}{t}\frac{\partial}{\partial_{a_j}}\biggr).
$$
Let $h(x)$ be a test function on $V$.  Then, when  $t\to \infty$, Guillemin-Sternberg state that
\begin{equation}\label{eq:Guillemin-Sternberg}
 \frac{1}{t^d} \sum_{x\in \c\cap \lattice} h(\frac{x}{t})\sim \left(Todd \biggl(\frac{1}{t}\frac{\partial}{\partial_a}\biggr)\cdot \int_{\c(a)} h(x) dx\right)|_{a=0}.
\end{equation}
As
$\c(a)=\sum_j a_j v_j + \c$,
 one sees easily that  (\ref{eq:Guillemin-Sternberg}) is precisely the Fourier transform of (\ref{eq:EMLasymptotic-anycone}).

\subsection{ Local Euler-Maclaurin formula}\label{subsec:localsimplicial}

Our method in the present article is to give expressions of type (\ref{eq:BB}), in terms of derivatives normal to the faces with respect to a given Euclidean scalar product.   The method is much simpler for simplicial cones, so we  sketch it now in this case. We will discuss directly the non simplicial case in the next section, and the result of this subsection are not needed. However, it might be useful to understand  this case.

Let $(v_1,v_2,\ldots,v_d)$ be a basis of $V$.
For any subset $J$ of $\{1,2,\ldots,d\}$, let $L_J$ be the subspace of $V$ spanned by $(v_j, j\in J)$ and let $C_J\subseteq V$ be its orthogonal component (for the given scalar product).

Let $K$ be a subset of $\{1,2,\ldots,d\}$. It is easy to see that
any rational function  of the form
$R(\xi)=\frac{P(\xi)}{\prod_{j\in K} \langle\xi,v_j\rangle}$
 can be written in a unique way as $R(\xi)=\sum_{J\subseteq K} \frac{P_J(\xi)}{\prod_{j\in J}\langle\xi,v_j\rangle}$ with $P_J\in \Sym(C_J)$, the symmetric algebra of $C_J$.
Indeed, write $P=Q+\sum _{j\in K} Q_j v_j$, with $Q\in \Sym(C_K)$. Then
$R= \frac{Q}{\prod_{j\in K}v_j}+ \sum_{j\in K} \frac{Q_j}{\prod_{k\in K,k\neq j}v_k}$.
Then iterate on
each term $\frac{Q_j}{\prod_{k\in J,k\neq j}v_k}$.
The uniqueness of the $P_J$'s  is also easily proved by induction on $|K|$, by taking  partial residues.

Now if $\c$ is a simplicial cone, with generators $v_1,\ldots, v_\ell$, any homogeneous component
$S(\c)_{[k-\ell]}(\xi)$ is of the form $\frac{P(\xi)}{\prod_{j=1}^\ell \langle\xi,v_j\rangle}$. By the preceding discussion,
$S(\c)_{[k-\ell]}(\xi)=\sum_{J\subseteq \{1,\ldots,\ell\}} \frac{P_{k,J}}{\prod_{j\in J} v_j}$ with $P_{k,J}\in \Sym(C_J)$.
For any $J\subseteq \{1,2,\ldots,\ell\}$, the cone generated by $(v_j,j\in J)$  is a face $\f$ of $\c$ and $I(\f)$ is equal to
$(-1)^{|J|} \frac{|\det(v_j,j\in J)|}{\prod_{j\in J} v_j}.$
Gathering the homogeneous terms face by face, we obtain the following decomposition of
$S(\c)(\xi)=\sum_{k\geq -\ell}S(\c)_{[k]}(\xi)$, indexed by all the faces $\f$ of $\c$.
\begin{lemma}\label{lem:musimplicial}
$$S(\c)(\xi)=\sum_\f M_\f(\xi) I(\f)(\xi)$$
where $M_\f$ is a holomorphic function of $\xi$
depending only of the orthogonal projection  of $\xi$ on the subspace $\f^{\perp}$ of $V^*$.
\end{lemma}

\begin{remark}
Given our Euclidean scalar product, the decomposition of $S(\c)(\xi)$ in Lemma \ref{lem:musimplicial} is unique.
 Thus the holomorphic term $M_\f$ is
the $\mu$-function of the transverse cone $\t(\c,\f)$ to the face $\f$ defined in
\cite{MR2343137}.
In particular, the term corresponding to $\f=\{0\}$ is the $\mu$-function of the cone $\c$.
In the appendix (Theorem \ref{th:localEML}), we establish this formula directly for any cone, while in
\cite{MR2343137}, we showed a valuation property  for the $\mu$ function allowing us to deduce the case of a general cone from the case of a simplicial one.
\end{remark}

Write $M_\f =\sum M_{m,\f}$ as a sum of its homogeneous terms.
We obtain the following corollary.
\begin{lemma}\label{lem:asympsimplicial}
Consider an Euclidean scalar product on $V$. If $\c$ is a simplicial cone,
 there exists unique constant coefficients differential operators $M_{m,\f}(\frac{\partial}{\partial x})$,
 homogeneous of degree $m$,  involving only   derivatives normal to the face $\f$ such that
 $$
\frac{1}{t^\ell}\sum_{x\in \c \cap \lattice} h(\frac{x}{t})\sim
\sum_{k=0}^\infty  \frac{1}{t^k}
\sum_{\stackrel{m\geq 0, \f\in \CF(\c)}{ m+\ell-\dim\f=k} }\int_\f M_{m,\f}(\frac{\partial}{\partial x})\cdot h(x)\, dm_\f(x).
$$
\end{lemma}

In the next section, we will show that such an asymptotic formula is valid for any affine cone,  with differential operators $M_{m,\f}(\frac{\partial}{\partial x})$ expressed in terms of the $\mu$-function \cite{MR2343137} of the transverse cone to the face $\f$.

\section{Local Euler-Maclaurin asymptotic expansion for Riemann sums over a  lattice polyhedron}\label{section:lattice-polyhedron}

\subsection{$\mu$ function of a pointed cone. }\label{section:mu}

We fix a Euclidean scalar product $Q$ on $V$. Thus $V^*$ inherits also a scalar product.
In \cite{MR2343137},  given  the scalar product $Q$,  we defined an analytic function $\mu_Q(\a)(\xi)$ on $V^*$ for any semi-rational affine cone $\a$ in a rational quotient space $V/L$ of $V$. The function
$\mu_Q(\a)(\xi)$  can also be defined by a renormalization procedure which is more natural in some respects than our inductive definition  in \cite{MR2343137}.
 Postponing this new definition (Definition \ref{def:mu-by-renormalization}) to the appendix, we note some of its properties.
In this section, we need only the case of a cone with a lattice vertex, the results for a semi-rational affine cone will be recalled in Section \ref{section:Dilatation by a real parameter}. We will often drop the subscript $Q$.

\begin{proposition}\label{prop:mu}
 \begin{enumerate}
 \item If $s$ is a lattice point, then $\mu(s+\c)=\mu(\c).$
 \item If $\a$ is an affine  cone in a quotient space $V/L$, then $\mu(\a)(\xi)$ depends only on the orthogonal projection of $\xi$ on the subspace $L^\perp\subset V^*$.
 \end{enumerate}
\end{proposition}

\begin{example}\label{ex:mu-dim1-1}
In dimension one, $V=\R$ and $\lattice=\Z$, for the cone $\R_{\geq 0}$, we have
\begin{equation}\label{eq:mu-dim1}
\mu(\R_{\geq 0})(\xi)=\frac{1}{1-\e^{\xi}}+\frac{1}{\xi}=- \sum_{m\geq 0}\frac{b_{m+1}}{(m+1)!}\xi^m.
\end{equation}
\end{example}

Recall the  definition of the {\em transverse cone} $\t(\c,\f)$ of $\c$ along a face $\f$ (Definition \ref{def:transverse-cone}).
The following theorem  is proven in  \cite{MR2343137} (a new proof is given in the appendix).
It generalizes Lemma \ref{lem:musimplicial} to any cone, not necessarily simplicial.

\begin{theorem}\label{th:local-EML-cone}
Let $\c\subset V$ be a  pointed cone. Fix a scalar product $Q$ on $V$.
Then
 \begin{equation}\label{eq:local-EML-cone-1}
    S(\c)= \sum_\f \mu_Q(\t(\c,\f)) I(\f),
\end{equation}
where the sum runs over the set of faces of $\c$.
\end{theorem}

More concretely, for each face $\f$,  write $V=L\oplus L^{\perp}$ where $L$ is the linear span of $\f$ and $L^{\perp}\subset V$
the orthogonal complement of $L$ in $V$ with respect to our scalar product.
Write $\xi\in V^*$ as $\xi=\xi_0+ \xi_1$ with
$\xi_0\in L^*$ and $\xi_1\in (L^{\perp})^*$,
then the term $\mu_Q(\t(\c,\f))$  is a holomorphic function of $\xi_1$,
 while $I(\f)$ is a rational function of $\xi_0$.
Equation (\ref{eq:local-EML-cone-1}) determines uniquely the functions
$\mu_Q(\t(\c,\f))(\xi_1)$.

For the case of a simplicial cone, we have seen that is was quite immediate to obtain such a decomposition.

\subsection{Reduction to pointed polyhedra}
As we continue, we will need to work with the supporting cones of a polyhedron, which are non pointed affine cones.  Actually, we may consider more generally a non pointed polyhedron $\p$ as well.
To begin with, the case  $\p=\R$. We write the Euler-Maclaurin formula for a Riemann sum over the whole line. For any $n>0$,
\begin{equation}\label{eq:EML-whole-line}
    \frac{1}{t}\sum_{x\in \Z}h(\frac{x}{t})= \int_{-\infty}^\infty h(x) - \frac{1}{t^n}\int_{-\infty}^\infty  \frac{B_n(\{ -tx\})}{n!}h^{(n)}(x).
\end{equation}

Since the polyhedron  $\p$ is rational,    its lineality space  $L$ is a rational subspace, meaning that $L\cap \lattice$ is a lattice $\lattice_L$. Let $\pi_{V/L}$ be the projection map $V\mapsto V/L$. Then the projected polyhedron $\pi_{V/L}(\p)$ is   pointed  in the  quotient space $V/L$; it is rational with respect to  the projected lattice $\lattice_{V/L}=\pi_{V/L}(\lattice)$.
\begin{lemma}\label{lemma:averaging}
  Let $\p\subset V$ be a polyhedron of dimension $\ell$ with lineality space $L$. If $h(x)$ is a test function on $V$, let $(\pi_{V/L})_*h (y)$ be the  function on $V/L$ obtained by averaging $h$ (with respect to the Lebesgue measure corresponding to the lattice $\lattice_L$),
  $$
  (\pi_{V/L})_*h (y)=\int_L h(y+x)dx.
  $$
  Then
  $$
  \frac{1}{t^\ell}\sum_{x\in t \p \cap \lattice}h(\frac{x}{t})=\frac{1}{t^{\ell-\dim L}}\sum_{y\in  t \pi_{V/L}(\p)\cap \lattice_{V/L}}(\pi_{V/L})_*h (\frac{y}{t})+ O(t^{-\infty}).
  $$
\end{lemma}
\begin{proof}
   If $\p=L$, then
   $\frac{1}{t^{\ell}}\sum_{x\in  \lattice}h(\frac{x}{t})=\int_{L } h(x) dx +O(t^{-\infty})$.  This follows immediately from (\ref{eq:EML-whole-line}).

In the general case, we obtain a product situation by choosing a complementary rational subspace $ L' $   to $L$,  $V=L\oplus L'$.
We identify $V/L$ and $L'$.
Let $\lattice'=(\Lambda+L)\cap L'$ be the projected lattice  on $L'$. Let $\p'\subset L'$ be the projected polyhedron. Then $\p'$ is a pointed polyhedron. We have $\p=L\oplus \p'$ and $\p\cap \lattice= L\cap \lattice \oplus \p'\cap \lattice'$. So the lemma follows from the case when $\p$ is  a vector space.
\end{proof}

\subsection{Local Euler-Maclaurin asymptotic expansion  for a cone}
 Let us state the following asymptotic expansion which we call the \emph{local Euler-Maclaurin asymptotic expansion} of the weighted sum on a cone. It depends on the choice of a scalar product on the ambient space.
 \begin{theorem}\label{cor:localEMLasymptotic-cone}
 Let $\c\subseteq V $ be a   rational   cone of dimension $\ell$. Let $\CF(\c)$ be the set of faces of $\c$.   Let $h(x)$ be a test function on $V $. Then the following asymptotic expansion holds for $t\to \infty$, $t$ real.
\begin{equation}\label{eq:EMLasymptotic-cone}
\frac{1}{t^\ell}\sum_{x\in \c \cap \lattice} h(\frac{x}{t})\sim
\sum_{k=0}^\infty  \frac{1}{t^k}
\sum_{\stackrel{m\geq 0, \f\in \CF(\c)}{ m+\ell-\dim\f=k} }\int_\f \mu(\t(\c,\f))_{[m]}(\frac{\partial}{\partial x})\cdot h(x)\, dm_\f(x).
\end{equation}

Here the constant coefficient differential operator
$\mu(\t(\c,\f))_{[m]}(\frac{\partial}{\partial x})$ involves only derivatives normal to the face $\f$.

\end{theorem}

\begin{remark}
    There is just one term for $k=0$, which is of course $\int_\c h(x) dx$. Indeed, for $k=0$, since $m\geq 0$, if  $m+\ell=\dim\f$, then $m=0$ and $\f=\c$. Moreover, for the face $\f=\c$ itself, we have $\t(\c,\c)=\{0\}$ and $\mu(\{0\})=1$, so this face occurs only for $k=0$.
\end{remark}
\begin{proof}
First, assume that $\c$ is pointed. In that case,  we   take the inverse Fourier transform of (\ref{eq:EMLasymptotic-anycone}),   collect the homogeneous components in Formula (\ref{eq:local-EML-cone-1}), noting that $I(\f)$ is homogeneous of degree $(-\dim\f)$, and apply Proposition \ref{prop:fourier-transform} (i). Thus we obtain (\ref{eq:EMLasymptotic-cone}) when $\c$ is pointed.

If $\c$ is not pointed, let $L$ be its lineality subspace and let $\pi= \pi_{V/L}$ be the projection map. Then (\ref{eq:EMLasymptotic-cone}) holds for the pointed cone $\pi(\c)$. We write it  for the averaged test function $\pi_* h$. The faces of $\pi(\c)$ are the projections of the faces of $\c$. If $\f$ is a face of $\c$, the transverse cone $\t(\pi(\c),\pi(\f))$ coincides with $\t(\c,\f)$ under the identification $(V/L)/\lin(\pi(\f))= V/\lin(\f)$. Thus,
\begin{multline}\label{eq:EMLasymptotic-cone-2}
 \frac{1}{t^{\ell-\dim L}}\sum_{y\in \pi(\c)\cap \pi(\lattice)}(\pi_*h)(\frac{y}{t})\sim
 \sum_{k=0}^\infty  \frac{1}{t^k}\\
\sum_{\stackrel{\{(\f,m)\}, \, \f\in \CF(\c)}{m\geq 0, \, m+(\ell-\dim L)-(\dim\f-\dim L)=k} }\; \int_{\pi(\f)} \mu(\t(\c,\f))_{[m]}(\frac{\partial}{\partial y})\cdot (\pi_*h)(y)\, dm_{\pi(\f)}(y).
\end{multline}
Since $\mu(\t(\c,\f))_{[m]}(\frac{\partial}{\partial y})$ is a differential operator with constant coefficients, we have
$$
  \mu(\t(\c,\f))_{[m]}(\frac{\partial}{\partial y})\cdot (\pi_* h)=
  \pi_*\left( \mu(\t(\c,\f))_{[m]}(\frac{\partial}{\partial x})\cdot h \right).
$$
So, the $(k,\f,m)$ term in the RHS of (\ref{eq:EMLasymptotic-cone-2}) is equal to
$$
\int_\f \mu(\t(\c,\f))_{[m]}(\frac{\partial}{\partial x})\cdot h(x)\, dm_\f(x).
$$
Using Lemma \ref{lemma:averaging}, we obtain (\ref{eq:EMLasymptotic-cone}) for the cone $\c$.
\end{proof}
\begin{example}
    If $\c=V$, there is just one face $\c$ itself, with transverse cone $\{0\}$, for which $\mu(\{0\})=1$, so all terms of the asymptotic expansion are $0$, except for $k=0$ which gives the integral over $V$.   \end{example}
\begin{example}\label{ex:localEMLasymptotic-cone-1}
 Let us describe (\ref{eq:EMLasymptotic-cone}) in dimension one, $\c=\R_{\geq 0}$. The face  $\f=\{0\}$ occurs in the $k$ term  for $m=k-1$.  We have $\t(\c,\{0\})=\c=\R_{\geq 0}$ and
$\mu(\R_{\geq 0})_{[k-1]}= -\frac{b_k}{k!} \xi^{k-1}$. So (\ref{eq:EMLasymptotic-cone}) is
$$
\frac{1}{t}\sum_{x\in \Z_{\geq 0}} h(\frac{x}{t})\sim \int_0^\infty h(x)  dx - \sum_{k=1}^\infty  \frac{1}{t^k}\frac{b_k}{k!}h^{k-1}(0).
$$
Of course, it is Formula (\ref{eq:EML-dim1-usual}) applied to $\frac{1}{t}h(\frac{x}{t})$.
\end{example}
\begin{example}
 Let us describe (\ref{eq:EMLasymptotic-cone}) for the non pointed cone $\c=\R_{\geq 0}e_1\oplus \R e_2\subset \R^2$, with the standard scalar product on $\R^2$. Besides $\c$ itself, there is only one face, $\f=\R e_2$, with transverse cone $\t(\c,\f)=\R_{\geq 0}e_1$ and
$\mu_Q(\t(\c,\f))_{[k-1]}(\xi)= -\frac{b_k}{k!} \xi_1^{k-1}$. So (\ref{eq:EMLasymptotic-cone}) is
\begin{multline*}
\frac{1}{t^2}\sum_{x_1\in \Z_{\geq 0}, x_2\in \Z} h(\frac{x_1}{t},\frac{x_2}{t}) \sim \\
\int_0^\infty \int_\R h(x_1,x_2)  dx_2 dx_1 -  \sum_{k=1}^\infty  \frac{1}{t^k}\frac{b_k}{k!}\int_\R \frac{\partial^{k-1}h}{\partial x_1}(0,x_2)dx_2.
\end{multline*}
It is Example \ref{ex:localEMLasymptotic-cone-1} applied to the averaged function $g(x_1)=\int_\R h(x_1,x_2)  dx_2$.
\end{example}
\begin{example}\label{ex:pyramide}
Let $V=\R^3$ with standard lattice.
Let $\c\subset \R^3$ be the non simplicial cone based on a square, with generators $v_1=e_3+e_1,v_2=e_3+e_2,v_3=e_3-e_1,v_4=e_3-e_2$. By subdivising $\c$ into unimodular cones, we obtain
$$
S(\c)(\xi)= {\frac {1+{ e^{\xi_{{3}}}}-{e^{2\,\xi_{{3}}}}-{e^{3\,
\xi_{{3}}}}}{ \left( 1-{e^{\xi_{{3}}+\xi_{{1}}}} \right)
 \left( 1-{e^{\xi_{{3}}-\xi_{{1}}}} \right)  \left( 1-{e^{
\xi_{{3}}+\xi_{{2}}}} \right)  \left( 1-{e^{\xi_{{3}}-\xi_{{2}}}
} \right) }}.
$$
We compute the first four homogeneous components of $S(\c)(\xi)$ from the formula  $S(\c)= \sum_\f \mu_Q(\t(\c,\f)) I(\f)$ (where $Q$ is the standard Euclidean product).
We obtain
$$
S(\c)_{[-3]}=I(\c)= -\frac{4\xi_3}{(\xi_3+\xi_1) (\xi_3-\xi_1)(\xi_3+\xi_2)(\xi_3+\xi_2)}.
$$
\begin{multline*}
S(\c)_{[-2]}=
\frac {1}{ \left( {\xi_3}+{\xi_1} \right)  \left( {\xi_3}
+{\xi_2} \right) }+\frac {1}{ \left( {\xi_3}+{\xi_1}
 \right)  \left( {\xi_3}-{\xi_2} \right) } \\
 +\frac {1}{ \left( {\xi_3}-{\xi_1} \right)  \left( {\xi_3}+{\xi_2
} \right) }
+\frac {1}{ \left( {\xi_3}-{\xi_1} \right)
 \left( {\xi_3}-{\xi_2} \right) }.
\end{multline*}
\begin{multline*}
S(\c)_{[-1]}= -\frac{2}{9}\sum_{j=1,2}\frac{1}{{\xi_3}+{\xi_j}}
+\frac{1}{{\xi_3}-{\xi_j}}
-\frac{1}{36}{\frac {{\xi_3}-{\xi_1}-{\xi_2}}{ \left( {\xi_3}+{\xi_1
} \right)  \left( {\xi_3}+{\xi_2} \right) }}\\
-\frac{1}{36}\,{\frac {{\xi_3}-{\xi_1}+{\xi_2}}{ \left( {\xi_3}+{\xi_1} \right)
 \left( {\xi_3}-{\xi_2} \right) }}
 -\frac{1}{36}\,{\frac {{\xi_3}+{
\xi_1}-{\xi_2}}{ \left( {\xi_3}-{\xi_1} \right)  \left( {
\xi_3}+{\xi_2} \right) }}-\frac{1}{36}\,{\frac {{\xi_3}+{\xi_1}+{
\xi_2}}{ \left( {\xi_3}-{\xi_1} \right)  \left( {\xi_3}-{
\xi_2} \right) }}.
\end{multline*}
\begin{multline*}
S(\c)_{[0]}=
\frac{1}{24}\,{\frac {{\xi_3}-{\xi_1}}{{\xi_3}+{\xi_1}}}+\frac{1}{24}\,{
\frac {{\xi_3}-{\xi_2}}{{\xi_3}+{\xi_2}}}
+\frac{1}{24}\,{\frac {{
\xi_3}+{\xi_1}}{{\xi_3}-{\xi_1}}}+\frac{1}{24}\,{\frac {{\xi_3
}+{\xi_2}}{{\xi_3}-{\xi_2}}}+\frac{1}{6}.
\end{multline*}

\medskip

Thus   the coefficients of the asymptotic expansion $$  \langle R_t(\c) ,h\rangle  \sim  \sum_{k=0}^\infty \frac{1}{t^k} \langle F_k,h\rangle$$ are
$$
\langle F_0,h\rangle =\int_\c h,\;\; \langle F_1,h\rangle= \frac{1}{2} \int_{\partial \c} h,
$$
$$
\langle F_2,h\rangle =
\frac{2}{9} \int_{\rm edges} h-\frac{1}{12}\int_{\partial \c} \partial_u h,\langle F_3,h\rangle =\frac{1}{6} h(0,0,0)-\frac{1}{24}\int_{\rm  edges}\partial_v h,
$$
where for each facet, $u$  is the primitive vector (for the projected lattice) normal to the facet pointing inwards,
and  $v=e_3-e_i$ for the edge $e_3+e_i$,
$v=e_3+e_i$ for the edge $e_3-e_i$,$i=1,2$.
\end{example}
\subsection{Local Euler Maclaurin asymptotic expansion for a lattice  polyhedron}
In this section, we will obtain an expansion similar to Theorem \ref{cor:localEMLasymptotic-cone} first for an affine cone with lattice vertex, then for a polyhedron. Thus, in the remainder of this  section, we assume that $\p$  is a  \emph{lattice polyhedron}, meaning that  each of its  faces  contains a lattice point. We also want $t\p$ to be a lattice polyhedron so we will restrict to $t\in \N$.

If $\p$ is a lattice polyhedron, the transverse cone along a face is a pointed affine cone with lattice vertex.  Theorem \ref{th:local-EML-cone} extends readily to such cones.
If $s\in \lattice$ and $\c$ is a pointed cone, for  the shifted cone  $\a=s+\c$,  we have
$$
S(s+\c)(\xi)=\e^{\langle\xi,s\rangle}S(\c)(\xi).
$$
The $\mu$ function of an affine cone $s+\c$, with $s\in \Lambda$, is equal to $\mu(\c)$.

For every face $\f$ of $\c$, we have  $\t(s+\c,s+\f)= \proj_{V/\lin(f)}(s)+ \t(\c,\f),$  therefore
$$
 \mu(\t(s+\c,s+\f))=\mu(\t(\c,\f)).
 $$
Moreover $I(s+\f)(\xi)=\e^{\langle\xi,s\rangle}I(\f)(\xi)$, so Formula (\ref{eq:local-EML-cone-1}) still holds if we replace $\c$ by an  affine cone $\a$ with lattice vertex.

If we dilate $\a$ by a positive integer $t$, the vertex of $t\a$ is still a lattice point. Therefore, we have the following extension of  Theorem \ref{cor:localEMLasymptotic-cone}  to the case of an affine \emph{lattice} cone (such that each of its faces contains a lattice point;  we do not assume that it is pointed), provided the parameter $t$ is an integer.
This theorem was proved by T. Tate (\cite{MR2737411}).
\begin{theorem}\label{cor:affine-lattice-cone}
 Let $\a\subseteq V $ be an    affine lattice cone of dimension $\ell$. Let $\CF(\a)$ be the set of faces of $\a$.   Let $h(x)$ be a test function on $V $. Then the following asymptotic expansion holds when $t\to \infty$ with $t\in \N$.
$$
\frac{1}{t^\ell}\sum_{x\in t\a \cap \lattice} h(\frac{x}{t})\sim
\sum_{k=0}^\infty  \frac{1}{t^k}
\sum_{\stackrel{m\geq 0, \f\in \CF(\a)}{ m+\ell-\dim\f=k} }\int_\f \mu(\t(\a,\f))_{[m]}(\frac{\partial}{\partial x})\cdot h(x)\, dm_\f(x).
$$
\end{theorem}
\begin{proof}
 We can write $\a=s+\c$ where $\c$ is a cone and $s\in \lattice$.  We apply Theorem  \ref{cor:localEMLasymptotic-cone} to the cone $\c$ and the shifted function $h(s+x)$.
\end{proof}
We will now show that this theorem leads to an asymptotic expansion for any lattice polyhedron. This theorem was obtained by
Tetsuya Tate \cite{MR2737411} (for a lattice polyhedron).
\begin{theorem}\label{th:main}
     Let $\p \subset V$ be a lattice polyhedron of dimension $\ell$. Let $h(x)$ be a  test function  on $V$. Fix a Euclidean scalar product $Q$ on $V$.  For a face $\f$ of $\p$, let $\t(\p,\f)$ be the transverse cone of $\p$ along $\f$.      Then the following asymptotic expansion  holds when  $t\to +\infty$, $t\in \N$,
\begin{multline}\label{eq:asymptotic}
\frac{1}{t^\ell}\sum_{x\in t\p\cap \lattice}h(\frac{x}{t})\sim \\
\int_\p h(x)dx +\sum_{k\geq 1}\frac{1}{t^k}\sum_{\stackrel{m\geq 0, \f\in \CF(\p)}{ m+\ell-\dim\f=k} }
\int_\f (\mu(\t(\p,\f)_{[m]}(\frac{\partial}{\partial x}))\cdot h)(x)dm_\f(x).
\end{multline}

Moreover, the differential operators which appear in (\ref{eq:asymptotic}) are unique in the sense that, if $\rho_{[m]}(\p,\f)(\xi)$ is a family of polynomials on $V^*$, such that (\ref{eq:asymptotic})  holds for any test function $h(x)$ (with  $\rho_{[m]}(\p,\f)$ in place of $\mu(\t(\p,\f)_{[m]}$) and such that $\rho_{[m]}(\p,\f)(\xi)$  is homogeneous of degree $m$,  and depends only on the $Q$-projection of $\xi$ on $(\lin(\f))^\perp$, then
$$
\rho_{[m]}(\p,\f)(\xi)= \mu_{[m]}(\t(\p,\f))(\xi).
$$
\end{theorem}

\begin{proof}
First, we prove (\ref{eq:asymptotic}). Let $h(x)$ be a test function.  Using a partition of unity, we may assume that the support of $h$ is contained in an open set $U$ such that
$U\cap \p=U \cap C(\p,\q)$ for some face $\q$ of $\p$ (recall that $C(\p,\q)$ is the supporting cone of $\p$ along $\q$).  If $U$ is  small enough and convex, we observe that
the  set of faces $\CF(C(\p,\q))$ of  $C(\p,\q)$ is in $1$-$1$ correspondence with the set of faces of $\p$ which meet $U$. If  $\d$ is such a face,
 the corresponding face $\f$ of $C(\p,\q)$ is  $ \f= \langle \d\rangle\cap\suppcone(\p,\q)$, where $ \langle \d\rangle$ denotes the affine  span of $\d$. Moreover, $\t(C(\p,\q),\f)= \t(\p,\d)$.
Thus Theorem \ref{cor:localEMLasymptotic-cone} gives
\begin{multline}\label{eq:supporting-cone-expansion}
 \frac{1}{t^\ell}\sum_{x\in t \p\cap \lattice}h(\frac{x}{t}) = \frac{1}{t^\ell}\sum_{x\in tC(\p,\q)}h(\frac{x}{t})\sim \\
 \sum_{k=0}^\infty  \frac{1}{t^k}
\sum_{\stackrel{\{(\d,m)\}, \, \d\in \CF(\p),\, \q\subseteq \d}{m\geq 0, \, m+\ell-\dim\d=k} }\;
 \int_{\langle\d\rangle \cap C(\p,\q)}\mu(\t(\p,\d))_{[m]}(\frac{\partial}{\partial x})\cdot h(x)\, dm_\d(x).
\end{multline}
 So the right hand side  of (\ref{eq:supporting-cone-expansion}) is
$$
 \sum_{k=0}^\infty  \frac{1}{t^k}
\sum_{\stackrel{\{(\d,m)\}, \, \d\in \CF(\p)}{m\geq 0, \, m+\ell-\dim\d=k} }\;
 \int_{\d}\mu(\t(\p,\d))_{[m]}(\frac{\partial}{\partial x})\cdot h(x)\, dm_\d(x).
$$
This is (\ref{eq:asymptotic}), except that the faces of $\p$ are labeled $\d$ instead of $\f$.

Let us prove the uniqueness part of the theorem. By uniqueness of asymptotic expansions, each $k$ term is a uniquely determined distribution
$$
\sum_{\stackrel{m\geq 0, \f\in \CF(\p)}{ m+\ell-\dim\f=k} }
\int_\f (\mu(\t(\p,\f)_{[m]}(\frac{\partial}{\partial x}))\cdot h)(x)dm_\f(x).
$$
By decreasing induction on $\dim \f$,  it follows that each $(\f,m)$ term is a uniquely determined distribution
$$
\int_\f (\mu(\t(\p,\f)_{[m]}(\frac{\partial}{\partial x}))\cdot h)(x)dm_\f(x).
$$
Indeed,  if $\f$ is a facet, we can restrict the support of $h$ so that it does not meet the other facets, etc.

Furthermore, as
$\mu(\t(\p,\f)_{[m]}(\frac{\partial}{\partial x})$ contains only  derivatives which are orthogonal to $\f$, it is  uniquely determined.
\end{proof}

\begin{remark}
For simplicial polytopes, one can just use  Lemma \ref{lem:asympsimplicial} instead of Theorem \ref{cor:affine-lattice-cone}.
\end{remark}

\subsection{Local behavior}
In order to illustrate the local behavior of the asymptotic expansion of Theorem  \ref{th:main}, we will compute the first terms for two triangles in the plane.

First we observe that the terms $F_1$ and $F_2$ are easily computed out of Theorem \ref{th:main} for any lattice polyhedron $\p$.

If $\p$ is lattice,  then $\mu(\t(\p,\p))= 1$, so the relation  $m+\ell-\dim \f=1$ is obtained exactly when $\f$ is a facet and $m=0$. Now, for a facet, we have  $\mu_{[0]}(\t(\p,\f))=- b_1=\frac{1}{2} $. This proves the formula $\langle F_1,h\rangle =\frac{1}{2}\int_{\partial \c}h$.

If $\p$ is a Delzant polyhedron, the next term $\langle F_2,h\rangle$ is also easy to compute with Theorem \ref{th:main}. For a facet $\f$, let $u_\f$ be the primitive generator of the transverse cone. For a face $\f$ of codimension $2$, let $(u^1_\f, u^2_\f)$ be the primitive edge generators of the transverse cone and let $C_\f= \frac{1}{4}+ Q(u^1_\f, u^2_\f) (\frac{1}{\|u^1_\f\|^2}+
\frac{1}{\|u^2_\f\|^2})$. Then, if $\p$ is  Delzant,
$$
\langle F_2,h\rangle =-\frac{1}{12}\sum_{\f,codim \f=1}\int_\f \partial_{u_\f}\cdot h+\sum_{\f,\codim \f=2}\int_\f C_\f h.
$$
If $\p$ is not Delzant, the two-dimensional transverse cones are still simplicial, though maybe  not unimodular, so the term $F_2$ involves arithmetic expressions (see \cite{MR2343137}).
\begin{example}\label{ex:local-behavior}
We write these formulas  for the following two triangles $\p$ and $\p'$.

   $\p$ is the triangle with vertices $(0,0)$, $(1,0)$,$(0,1)$. Let $\f_1$ be its horizontal edge, $\f_2$ the vertical one and $\f_3$ the diagonal one. Then
\begin{multline*}
\frac{1}{t^2}
\sum_{(x_1,x_2)\in t\p\cap  \Z^2}h(\frac{x_1}{t},\frac{x_2}{t})\sim
\int_\p h +\frac{1}{2t}\int_{\partial(\p)}h  +\\
 \frac{1}{t^2}\biggl(-\frac{1}{12}(\int_{\f_1}\partial_{x_2}h + \int_{\f_2}\partial_{x_1}h-\int_{\f_3}( \frac{\partial_{x_1} + \partial_{x_2} }{2})\cdot h )+
\\ \frac{1}{4}h(0,0)+\frac{3}{8}h(1,0)+\frac{3}{8}h(0,1)\biggr)+ \cdots.
\end{multline*}

$\p'$ is the triangle with vertices $(0,0)$, $(2,0)$,$(0,3)$, horizontal edge $\f'_1$, vertical edge $ \f'_2$,  diagonal edge $\f'_3$.
Then
\begin{multline*}
\frac{1}{t^2}
\sum_{(x_1,x_2)\in t\p'\cap  \Z^2}h(\frac{x_1}{t},\frac{x_2}{t})\sim
\int_{\p'} h +\frac{1}{2t}\int_{\partial(\p')}h +\\
\frac{1}{t^2}\biggl(
-\frac{1}{12}(\int_{\f'_1}\partial_{x_2}h + \int_{\f'_2}\partial_{x_1}h-\int_{\f'_3}(\frac{3\partial_{x_1} + 2\partial_{x_2} }{13})\cdot h )+\\ \frac{1}{4}h(0,0)+\frac{19}{52}h(2, 0)+ \frac{5}{13}h(0,3)\biggr)+ \cdots.
\end{multline*}
Indeed, the two expansions coincide if the  test function $h$ is supported near $(0,0)$.
\end{example}
\subsection{The remainder? }
For the remainder at order $k$ of the asymptotic expansion, a closed formula (in terms of derivatives of locally polynomial measures supported on faces) looks elusive.
As an example, we compute the remainder at order $k=1$ (Riemann sum minus integral) for the triangle $\p$ of Example \ref{ex:local-behavior}.
\begin{example}\label{ex:remainder}
Let $D_n=\frac{1}{2}(\partial_{x_1}+\partial_{x_2})$ be the normal derivative to the diagonal edge $\f_3$ and let
$D_t=\frac{1}{2}(\partial_{x_1}-\partial_{x_2})$ the tangential derivative.
For $t\in \N$, we have
 \begin{multline*}
 \frac{1}{t^2} \sum_{(x_1,x_2)\in t\p\cap  \Z^2}h(\frac{x_1}{t},\frac{x_2}{t})- \int_\p h(x_1,x_2)dx_1dx_2 \; =\\
\int_\p A(t) d{x_1}d{x_2}+\int_{\f_3} B(t) d m_{\f_3}+\frac{1}{t^2}h(1,0),
 \end{multline*}
where
\begin{multline*}
A(t)= \frac{1}{t}\left((\{t{x_1}\}-1) \partial_{x_1}h+ (\{t{x_2}\}-1)\partial_{x_2} h \right)\; + \\
\frac{1}{t^2} (\{t{x_1}\}-1)(\{t{x_2}\}-1)\partial^2_{{x_1}{x_2}}h ,
\end{multline*}
\begin{multline*}
 B(t)=\frac{3}{2t} h+\frac{1}{t^2} \{t {x_1}\}
 (\{t {x_1}\}-1)D_n\cdot h+
 \frac{1}{2t^2}
 (\{t x_1\}-1)D_t\cdot h.
\end{multline*}
\end{example}

\section{Riemann sums with real scaling parameter over semi-rational polyhedra}\label{section:Dilatation by a real parameter}
\subsection{Asymptotic expansions with step-polynomial coefficients}\label{section:step-poly-asymptotic}

 We recall now the  Euler-Maclaurin summation formula  on the half-line $[s,\infty[$, for any real $s$, for a test function $h(x)$, (\cite{MR2312338}, Theorem 9.2.2).
\begin{multline}\label{eq:EML-dim1-any-s}
    \sum_{x\in \Z, x\geq s}h(x)=
    \int_s^\infty h(x) dx -\sum_{k=1}^n  \frac{B_k(\{- s\})}{k!}h^{(k-1)}(s)\\
    - \int_s^\infty   \frac{B_n(\{- x\})}{n!}h^{(n)}(x)  \, dx.
\end{multline}
For any real $t>0$, let
$$
h_t(x)=\frac{1}{t}h(\frac{x}{t}).
$$
Substituting $h_t$ for $h$ and $ts$ for $s$ in this formula, and changing variables in the integrals on the right-hand-side, we obtain
\begin{multline}\label{eq:EML-dim1-asymptotic-vertex-s}
\frac{1}{t} \sum_{x\in \Z, x\geq ts}h(\frac{x}{t})=
 \int_s^\infty  h(x) dx -
 \sum_{k=1}^{n-1} \frac{1}{t^k} \frac{B_k(\{-t s\})}{k!}h^{(k-1)}(s)\\
 -
  \frac{1}{t^n}   \left(\frac{ B_n(\{-t s\})}{n!} h^{(n-1)}(s)+\int_s^\infty \frac{B_n(\{-t x\})}{n!} h^{(n)}(x)dx\right)
\end{multline}
This formula can be considered as an asymptotic expansion (with  a closed expression for the remainder) when $t\to +\infty$, $t\in \R$, where we allow the coefficient of $\frac{1}{t^k}$ to be  a polynomial function of the fractional part $\{-t s\}$, that is, a step-polynomial function in the sense of the following definition.

\begin{definition}\label{def:step-poly}
A \emph{step-polynomial} function $F(t)$ on $\R$ is an element of the algebra generated by the functions $t\mapsto \{\gamma t\}$, with $\gamma \in \R$.
\end{definition}
If $F(t)$ is a step-polynomial, there exists

1) a finite set $\Gamma$  of real numbers linearly independent over $\Q$.

2) for each $\gamma\in \Gamma$ a finite set of non zero integers $n_{\gamma, k}$

3) a polynomial in several variables, $P(X_{\gamma,k})$,

such that
\begin{equation}\label{eq:step-poly}
F(t)=P(\{n_{\gamma, k} \gamma t\}).
\end{equation}
An example of step-polynomial function with irrational  $\gamma$'s  is displayed in Figure \ref{fig:steppoly}.

A step polynomial function is defined for all $t\in \R$, and has a discrete set of discontinuities.

One must be aware that this set of data  is not uniquely defined by $F(t)$. A simple example is
$$
1-\{t\}-\{-t\}= (1-\{2t\}-\{-2t\})(1-\{3t\}-\{-3t\}).
$$

\begin{figure}
\begin{center}
 \includegraphics[width=10cm]{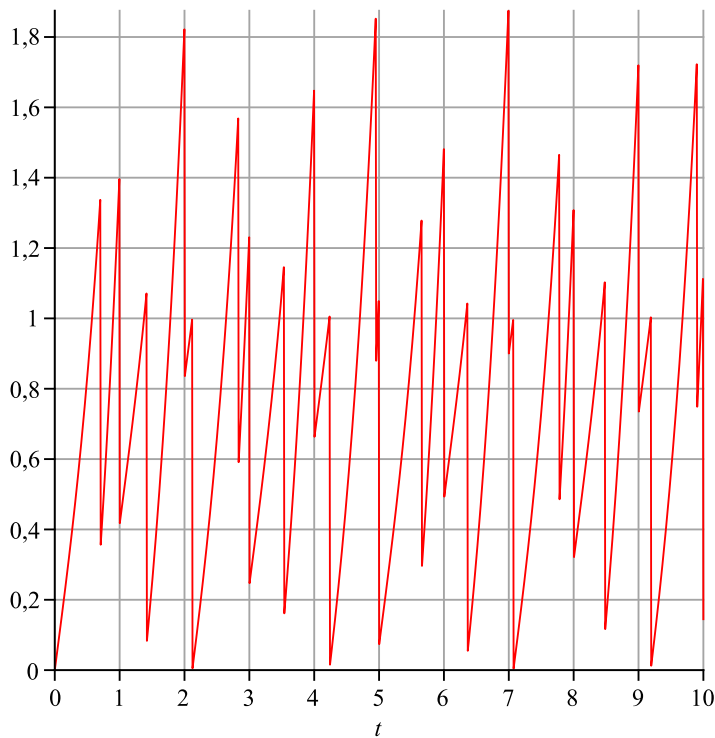}
 \caption{Graph of  $\{\sqrt2 t\}+\{t\}^3$}\label{fig:steppoly}
 \end{center}
\end{figure}

\begin{definition}\label{def:asymptotic1}
A function $\phi(t) $, defined for $t>0$, has an  \emph{asymptotic expansion with step-polynomial coefficients} $F_k(t)$, when $t\to +\infty$, $t\in \R$, written as
\begin{equation}\label{eq:asymptotic1}
 \phi(t) \sim F_0(t)+\frac{ 1}{t} F_1(t)+ \cdots + \frac{ 1}{t^n} F_n(t)+ \cdots,
\end{equation}
if
$ \phi(t) = F_0(t)+\frac{ 1}{t} F_1(t)+ \cdots + \frac{ 1}{t^n} F_k(t) + O(\frac{1}{t^{n+1}})$  when $t\to +\infty$, for every $n\geq 1$.
\end{definition}
For lack of a reference, we give a proof for the following result.
\begin{proposition}
 If such an asymptotic expansion exists, it is unique.
\end{proposition}
\begin{proof}
First we prove that if $F(t)$ is a step-polynomial function and $F(t)=O(\frac{1}{t})$ when $t\to +\infty$,  then $F(t)=0$ for all $t\in \R$.
If $F(t)$ has an expression with  a single $\gamma$,  then $F(t)$ is periodic, so it must be identically $0$.

Let us consider the case where $\Gamma$ consists of two elements, $1$ and $\gamma \notin \Q$.  For instance, assume that
$$
F(t)= P(\{t\},\{-5t\},\{\gamma t\},\{ 2 \gamma t\})=O(\frac{1}{t}).
$$
For  $t=N+ a$ with $N\in \N$ and $a$ fixed, we have
$$
F(N+a)=P(\{a\},\{-5a\},\{\gamma N + \gamma a\},\{ 2 (\gamma N +  \gamma a)\})=O(\frac{1}{N}).
$$
Fix $b $ such that $0<b<1$. By Kronecker theorem in one dimension, (\cite{MR2445243}, Theorem 438) there exists a strictly increasing sequence $N_m\in \N$ such that $\lim_{m\to \infty}\{\gamma N_m + \gamma a\}= b$.
As $ \{2 x\}= 2x$ if $x\in [0,\frac{1}{2}[$ and  $ \{2 x\}= 2x-1$ if $x\in [\frac{1}{2}, 1[$, we  conclude that $P(\{a\},\{-5a\},b, 2 b)=0$ for any $a$ and $b\in ]0,\frac{1}{2}[$,  hence for  any $b$ since $P$ is a polynomial, and similarly  $P(\{a\},\{-5a\},b, 2 b-1)=0$ for any $b$.  Therefore,
$P(\{a\},\{-5a\},\{b\},\{ 2 b\})=0$ for any $a$ and $b$, hence $F(t)$ is identically $0$.

A similar argument works in the general case. If $\Gamma = (\gamma_0, \gamma_1, \ldots, \gamma_p)$, we can assume that $\gamma_0=1$. By  Kronecker theorem, (\cite{MR2445243}, Theorem 442),  the set of points $(\{N\gamma_1+u_1\},\ldots, \{N\gamma_p+u_p\}), N\in \N$  is dense in $[0,1[^p$ for any $(u_1,\ldots, u_p)$.

Thus we have proved that $F(t)=0$. The proposition follows by induction on $k$ as usual. \end{proof}

\subsection{Step-polynomials on $V$. $\mu$ function of a semi-rational affine cone}
We fix a Euclidean scalar product $Q$ on $V$. We recall now the properties of the holomorphic function $\mu_Q(\a)$ when  $\a$ is a semi-rational affine cone (a rational  cone shifted by a real vertex, not just a rational one). We begin with some examples.
\begin{example}\label{ex:mu-dim1-1-semi-rational}
In dimension one, $V=\R$ and $\lattice=\Z$, for the cone $s+\R_{\geq 0}$, where $s\in \R$, we have
\begin{equation}\label{eq:mu-dim1-semi-rational}
\mu(s+\R_{\geq 0})(\xi)=\frac{\e^{\{-s\}\xi}}{1-\e^{\xi}}+\frac{1}{\xi}=- \sum_{m\geq 0}B_{m+1}(\{-s\})\xi^m.
\end{equation}
\end{example}
\begin{example}\label{ex:mu-dim1-2-semi-rational}
In  $\R^d$ with standard lattice and Euclidean scalar product, let $\c$ be the one-dimensional cone $ \c=\R_{\geq 0}e_1$, let $s=(s_1,\dots, s_d)$.
Then
\begin{align*}
&\mu(s+\R_{\geq 0}e_1)(\xi)\\
& =\frac{\e^{\{-s_1\}\xi_1}}{1-\e^{\xi_1}}+\frac{1}{\xi_1} =- \sum_{m\geq 0}B_{m+1}(\{-s_1\})\xi_1^m,  \mbox{  if }(s_2,\dots, s_d)\in \Z^{d-1},\\
&=0  \mbox{  otherwise }.
\end{align*}
 Therefore
\begin{equation}\label{eq:mu-dim1-2}
\mu(s+\R_{\geq 0}e_1)(\xi)= -\prod_{j=2}^d (1-\{s_j\}-\{-s_j\}) \sum_{m\geq 0}B_{m+1}(\{-s_1\})\xi_1^m.
\end{equation}
\end{example}
 We see that these functions involve step-polynomials in several variables.
\begin{definition}\label{def:step-poly-several-variables}
 A step-polynomial function on $V$ is an element of the algebra generated by the functions $s \mapsto \{\langle \gamma,s\rangle\}$, with $\gamma\in V^*$.  A \emph{rational step-polynomial} function is an element of the algebra generated by the functions $s \mapsto \{\langle\gamma, s\rangle\}$, with rational  $\gamma$'s.
\end{definition}

Step-polynomials occur when we consider the generating function of an affine  cone with  real vertex (\cite{SLII2014},Theorem 2.22). We recall the result.
Let $\c\subset V$ be a pointed cone and $s\in V$. Consider the \emph{shifted generating function}
\begin{equation}\label{eq:shifted-generating-function}
  M(s,\c)(\xi)= \e^{-\langle \xi,s \rangle}S(s+\c)(\xi).
\end{equation}

\begin{lemma}\label{lemma:S(s+c)with-step-polys}
Let $\c\subset V$ be a  pointed cone. Let $v_1, \ldots, v_n$ be the set of its lattice edge generators.
The homogeneous components of $M(s,\c)(\xi)$ are of the form $M(s,\c)_{[m]}(\xi)=\frac{P_{m+n}(s,\xi)}{\prod_j \langle \xi,v_j\rangle}$ where $P_{m+n}(s,\xi)$ is a polynomial of degree $m+n$ in $\xi$ with coefficients which are rational step-polynomial functions of $s$.
\end{lemma}

Let us now state the properties of the $\mu$-function of a pointed affine cone.
\begin{proposition} \label{prop:mu-semi-rational-cone} Let $\a\subset V$ be a semi-rational affine cone.

\noindent (i) For any $v\in \lattice$, we have $\mu(v+\a)=\mu(\a)$.

\noindent (ii) Let $s$ vary in $V$. For   every integer
$m \geq 0$,  the homogeneous component $\mu(s+\a)_{[m]}(\xi)$ is a polynomial function of $\xi\in V^*$  with coefficients which are step-polynomial functions of $s$.

\noindent (iii)  For  every integer
$m \geq 0$, the homogeneous component $\mu_{[m]}(t\a)(\xi)$ is a polynomial function of $\xi\in V^*$ with coefficients step-polynomial functions of $t\in \R$. If the scalar product $Q$ is rational, the coefficients are rational step-polynomials.
\end{proposition}
\begin{proof}
 We deduce  properties (i) and (ii) from the definition of $\mu(\a)$ by renormalization \ref{def:mu-by-renormalization}. (iii) follows immediately from (ii)

We have  $S(v+\a)=\e^{\langle \xi,v \rangle} S(\a)$  if $v\in \lattice$, hence (i).

(ii) is a consequence of Lemma \ref{lemma:S(s+c)with-step-polys}:  the renormalization $R_Q\bigl(\frac{P_{m+n}(s,\xi)}{\prod_j \langle \xi,v_j\rangle}\bigr)$ of $M(s,\c)_{[m]}(\xi)$ is also a polynomial in $\xi$ with coefficients which are step-polynomial functions of $s$.
\end{proof}

\subsection{ Local Euler-Maclaurin asymptotic expansion for a semi-rational polyhedron}
We can now state the local Euler-Maclaurin asymptotic  expansion for  Riemann sums over a semi-rational  polyhedron when the scaling parameter is  real, not only integral. Given the above  discussion on asymptotic expansions  with step-polynomial coefficients, the proof is parallel to the case of a lattice polyhedron and an integral scaling parameter in Section \ref{section:lattice-polyhedron}. So we will leave the details to the reader.

The main step is the following analogue of  Theorem \ref{th:asymptotic-cone} about  Riemann sums over a cone,  in the case  of a semi-rational affine cone and a real  scaling parameter.

\begin{theorem}\label{th:asymptotic-cone-real}
Let $\c\subset V $ be a  pointed cone of dimension $\ell$. Let $s\in V$. Let $\lambda\in V^*$ be such that $ -\lambda$ lies in the dual cone of  $\c$.
 Consider the  distribution on $V$ given by
\begin{equation}\label{eq:asymptotic-Fourier-realt}
\langle R_t(s +\c) ,h\rangle =\frac{1}{t^\ell}\sum_{x\in (ts+\c) \cap \lattice}h(\frac{x}{t}).
\end{equation}
It  has an asymptotic expansion with step-polynomial coefficients when $t\to \infty$, t real, the Fourier transform of which is given by
\begin{eqnarray}\label{eq:EMLasymptotic-anycone-real-t}
 \CF( R_t( s +\c) )(\xi)&=&\frac{1}{t^\ell} \lim\limits_\lambda (S(t s+\c)(-i \frac{\xi}{t}))  \\
\nonumber & =& \frac{1}{t^\ell} \e^{-i\langle \xi,s \rangle} \lim\limits_\lambda ( M(t s,\c) (-i \frac{\xi}{t}) ) \\
\nonumber &\sim&  \e^{-i\langle \xi,s \rangle}\sum_{k=0}^\infty \frac{1}{t^k}\lim\limits_\lambda (
 M(t s,\c)_{[k-\ell]}(-i\xi)).
\end{eqnarray}
\end{theorem}
\begin{proof}
    If $\c=\R_{\geq 0}$ and $s$ is any real number, (\ref{eq:EMLasymptotic-anycone-real-t}) follows from the classical dimension one Euler-Maclaurin formula (\ref{eq:EML-dim1-any-s}) for a half-line. The general case is reduced to dimension one by subdivision of cones, as in the proof of Theorem \ref{th:asymptotic-cone}.
\end{proof}

From there, we deduce the case of a polyhedron in a similar manner to Section \ref{section:lattice-polyhedron}.

\begin{theorem}\label{th:EML-local-asymptotic-real-dilatedpolyhedron} Let $V$ be a vector space with lattice $\lattice$. Fix a  Euclidean scalar product $Q$ on $V$.
  Let $\p \subseteq V$ be a semi-rational polyhedron of dimension $\ell$. For any test function  $h(x)$ on $V$, the following asymptotic expansion  with  step-polynomial coefficients holds when  $t\to +\infty$ ($t\in \R$),
\begin{multline}\label{eq:asymptotic-semi-rational}
\frac{1}{t^\ell}\sum_{x\in t\p\cap \lattice}h(\frac{x}{t})\sim \\
\int_\p h(x)dx +\sum_{k\geq 1}\frac{1}{t^k}\sum_{\stackrel{m\geq 0, \f\in \CF(\p)}{ m+\ell-\dim\f=k} }
\int_\f (\mu(t\t(\p,\f))_{[m]}(\frac{\partial}{\partial x}))\cdot h)(x)dm_\f(x).
\end{multline}
Moreover, the differential operators which appear in (\ref{eq:asymptotic-semi-rational}) are unique in the sense that, if $\rho_m(\p,\f,t)(\xi)$ is a family of polynomials on $V^*$ with step-polynomial coefficients of $t\in \R$, such that (\ref{eq:asymptotic-semi-rational})  holds for any test function $h(x)$ and such that $\rho_m(\p,\f,t)(\xi)$  is homogeneous of degree $m$  and depends only on the $Q$-projection of $\xi$ on $(\lin(\f))^\perp$, then
$$
\rho_m(\p,\f,t)(\xi)= \mu_{[m]}(t\t(\p,\f))(\xi).
$$
\end{theorem}
\begin{example}\label{ex:triangle-t-reel}
For the triangle $\p$ with vertices $(0,0)$, $(1,0)$,$(0,1)$, the asymptotic expansion with step-polynomial coefficients, for  $t\in \R$, is
\begin{multline*}
\frac{1}{t^2}
\sum_{(x_1,x_2)\in t\p\cap  \Z^2}h(\frac{x_1}{t},\frac{x_2}{t})\sim
\int_\p h +\frac{1}{t}\int_{\partial(\p)}(\frac{1}{2}-\{t\})h\; + \\
\frac{1}{t^2}\left(
-\frac{1}{12}(\int_{\f_1}\partial_{x_2}h + \int_{\f_2}\partial_{x_1}h)  +
(-\frac{1}{12}+\frac{\{t\}}{2}-\frac{\{t\}^2}{2})\int_{\f_3}
(-\frac{\partial_{x_1} + \partial_{x_2} }{2})\cdot h\right )\\ +\frac{1}{4}h(0,0)+\frac{3}{8}h(1,0)+\frac{3}{8}h(0,1))+ \cdots
\end{multline*}
When we restrict this to $t\in \N$, we recover the formula with constant coefficients of Example \ref{ex:local-behavior}.
\end{example}

\section{Appendix. Renormalization.}\label{section:appendix}
\subsection{Decomposition of the space of rational functions with poles in a  finite set of hyperplanes}
In this section which is elementary linear algebra, the base field need not be $\R$.

Let $\Delta=\{v_1,\dots, v_N\}$ be a finite set of  non zero elements of $V$. We assume that $\Delta$  does not contain collinear vectors. We denote by $\CR_\Delta$ the algebra of rational functions on $V^*$ of the form
$$
\frac{P(\xi)}{\prod_{j} \langle v_j,\xi \rangle^{n_j}}
$$
where $P(\xi)$ is a polynomial function on $V^*$ and $n_j\geq 0$.
We will denote this fraction simply by $\frac{P}{\prod_{j}  v_j^{n_j}}$. Its poles are contained in the union of the hyperplanes $\langle v_j,\xi \rangle=0$.

All statements and proofs of this appendix are quite easy when the   $v_j$'s are linearly independent.

\begin{definition}
If $L$ is a subspace of $V$, we denote by $G_{\Delta\cap L}$ the subspace of
$\CR_\Delta$ spanned by the fractions
$$
\frac{1}{\prod_{j\in J} v_j^{n_j}}
$$
 with $n_j>0$, $v_j\in L$, and  $\{v_j, j\in J\}$  span $L$.

 We denote by $B_{\Delta\cap L}$ the subspace of $G_{\Delta\cap L}$ spanned by  the fractions
 $\frac{1}{\prod_{j\in J} v_j}$  where $(v_j, j\in J)$ is a basis of $L$. Thus $B_{\Delta\cap L}$ is the homogeneous summand  of degree $-\dim L$ of $G_{\Delta\cap L}$. For $L=V$, we denote these spaces simply by $G_\Delta$ and $B_\Delta$.
\end{definition}
\begin{definition}
A \emph{complement map} is a map  $L\mapsto C(L)$ which associates
to any subspace $L\subset V$ a complementary subspace $C(L)$, so
that $V=L\oplus C(L)$.
\end{definition}
For example, a scalar product $Q$ defines a complement map.

We recall that the symmetric algebra $\Sym(V)$ is canonically identified with the algebra of polynomial functions on $V^*$ and the symmetric algebra of the dual, $\Sym(V^*)$, is canonically identified with the algebra of constant coefficients differential operators on $V^*$ (for clarity, we will denote by $\partial_\gamma$ the differentiation with respect to $\gamma\in V^*$). Given a complement  map,  if $L$ is a subspace of
$V$, we can consider $\Sym(C(L))$ as a subspace of $\Sym(V)$.
\begin{theorem}\label{th:structure}
 Let $C$ be a complement map on the set of subspaces of $V$.  We
 have a direct sum  decomposition
$$
\CR_\Delta=\oplus_{L\subseteq V } \Sym(C(L)) \otimes G_{\Delta\cap L}.
$$
\end{theorem}
Of course, the sum runs only over the set of subspaces which are spanned by elements of $\Delta$. The summand corresponding to $L=\{0\}$ is the space of polynomials $\Sym(V)$.
\begin{remark}
 If the base field is   $\R$, and  $\Delta$ is contained in an open half-space $\lambda> 0$, (this is realized by replacing $v_j$ by $-v_j$ if necessary),  then for  $R\in \CR_\Delta$, the algebraic decomposition of Theorem \ref{th:structure} can be translated into a decomposition of the inverse Fourier transform $\CF^{-1}(\lim\limits_\lambda R)$ into a sum of distributions each  supported on a subspace $L$. Thus the uniqueness part of the theorem is easy to prove, as in the proof of Theorem \ref{th:main}.
\end{remark}
\begin{example}
    If $\dim V =1$,  then $\Delta$ consists of just one element say $\langle v_1,\xi\rangle=\xi$,   there are two subspaces $L=V$ and $L=\{0\}$. Then $\CR_\Delta$ consists of fractions
    $\frac{f(\xi)}{\xi^n}$. Such a fraction can be written in a unique way as $\frac{a_n}{\xi^n}+ \cdots + \frac{a_1}{\xi}+ g(\xi)$, where $a_j$ are constants and $g(\xi)$ is a polynomial.
\end{example}
\begin{proof}
The result is certainly not new. It is implicit in \cite{MR2722776} and \cite{MR1710758}.  For completeness, we will give a full proof,  to begin with, the following elementary lemma.
\begin{lemma}\label{lemma:simple-fractions}
 $G_{\Delta\cap L}$ is spanned by fractions  of the form
 \begin{equation}\label{eq:GL}
\frac{1}{\prod_{j\in J}  v_j^{n_j}} \mbox{ where  } (v_j, j\in J) \mbox{  is a basis of } L.
\end{equation}
\end{lemma}

 For example, $\frac{1}{(v_1+v_2)v_1 v_2 }=\frac{1}{(v_1+v_2)^2 v_2}+\frac{1}{(v_1+v_2)^2 v_1}$.

\begin{proof}
It is enough to prove that the space spanned by fractions of the form \ref{eq:GL} is stable by multiplication by $\frac{1}{v^n}$ where $v\in \Delta\cap L$.

We decompose $v=\sum_{j\in J} c_j v_j$ in the basis $(v_j, j\in J)$ of $L$. Then
\begin{multline*}
\frac{1}{v^n \prod_{j\in J} v_j^{n_j}}= \frac{v}{v^{n+1} \prod_{j\in J} v_j^{n_j}}
=\sum_{j\in J} c_j \frac{1}{v^{n+1} v_j^{n_j-1}\prod_{k\in J,k\neq j} v_k^{n_k}}.
\end{multline*}
Pick  $j$  such that $c_j\neq 0$. If $n_j=1$, the $j$ term is of the form (\ref{eq:GL}), with $v_j$ replaced by $v$ in the basis $(v_j, j\in J)$ of $L$.
If $n_j>1$, we repeat the procedure on the $j$ term, with $v^n$ replaced by $v^{n+1}$, so the assertion follows by induction on $\sum_j n_j$.
\end{proof}

The equality
$\CR_\Delta=\sum _{L\subseteq V } \Sym(C(L)) \otimes G_{\Delta\cap L}$
follows easily from this lemma,   by using the decomposition
$\Sym(V)=\Sym(C(L))\otimes S(L)$. There remains to prove that the sum is direct. The main step is \cite{MR1710758}, Theorem 1, which we recall in the following  lemma.

We introduce the  subspace $N_\Delta\subset \CR_\Delta$ spanned by  fractions $\frac{P}{\prod_{j\in J}  v_j^{n_j}}$ where $(v_j, j\in J)$ \textbf{do not}  span $V$. Thus
$$N_\Delta=\sum_{L\subsetneqq V } \Sym(C(L)) \otimes G_{\Delta\cap L}.$$

\begin{lemma}\label{lemma:arrangementsI}
Assume that $\Delta$ spans $V$.  Then

\noindent(i)  $\CR_\Delta=  G_\Delta\oplus N_\Delta$.

\noindent(ii) $ N_\Delta$ is the torsion $\Sym(V^*)  $ submodule of $\CR_\Delta$.

\noindent(iii) $G_\Delta$ is a free $\Sym(V^*)  $ module.
\end{lemma}
\begin{proof}
Let us prove that $ N_\Delta$ is a torsion $\Sym(V^*)  $-module. Let $\phi=\frac{P}{\prod_{j\in J}  v_j^{n_j}}$, where $(v_j, j\in J)$ \textbf{do not}  span $V$.
There exists  $\gamma\neq 0$ such that $\langle\gamma,v_j\rangle=0$ for every $j\in J$. For $N$ large enough,   $\partial_{\gamma}^{N}\cdot \phi=\frac{\partial_{\gamma}^{N}\cdot P}{\prod_{j\in J}  v_j^{n_j}}=0$.

We prove the lemma by induction over $\dim V$.
We consider the subspace $B_\Delta$  of $G_\Delta$  spanned by fractions of the form
$$
\frac{1}{\prod_{j\in B}v_j}
$$
where $(v_j, j\in B)$ is a basis of $V$. It follows from Lemma \ref{lemma:simple-fractions} that  $G_\Delta$ is generated as $\Sym(V^*)$-module by  $B_\Delta$.

Let us first prove that $B_\Delta\cap N_\Delta=0$. We use the notion of partial residue. If $\phi\in B_\Delta$, the poles of $\phi$ are simple and contained in $\Delta$. Let us fix an element of $\Delta$, say $v_1$.  Let $V_0 =V/\R v_1$, thus $V_0^*= v_1^\perp$. Let $\Delta_0$ be the set of projections of the other elements of $\Delta$  on $V_0$. Thus $\CR_{\Delta_0}$  consists of functions on $V_0^*$.

\begin{definition}
If $v_1$ is a simple pole  of $\phi\in \CR_\Delta$, the residue $\Res_{v_1}\phi$ is the element of $\CR_{\Delta_0}$ defined by
$$
\Res_{v_1}\phi= (v_1\phi)|_{v_1^\perp}.
$$
\end{definition}
Let $\phi\in B_\Delta\cap N_\Delta$. It is clear that $\Res_{v_1}\phi \in B_{\Delta_0}$.
Let us show  that $\Res_{v_1}\phi=0$.
We write
$\phi=\sum_L g_L$ with $g_L\in \Sym(C(L)) \otimes G_{\Delta\cap L}$. Let $(L_a, a\in A)$ be the  subspaces which contains $v_1$. For each $a\in A$,  we pick $\gamma_a\in L_a^\perp$, $\gamma_a\neq 0$.  For $N$ large enough,   $\partial_{\gamma_a}^{N}\cdot g_{L_a}=0$. Let $D =\prod_{a\in A}\partial_{\gamma_a}^{N}$. Then $D\cdot g_{L_a}=0$ for every $a\in A$.

If $v_1\notin L$, then  $v_1$ is not a pole of $g_L$,  hence not a pole of $D\cdot g_L$ either. Therefore $v_1$ is not a pole of $D\cdot \phi$, so $\Res_{v_1}( D \cdot \phi)=0$.  The elements $\gamma_a$ are in $v_1^\perp=V_0^*$, so that $\prod_a {\gamma_a}^{N}\in \Sym(V_0^*)$, and
$$
D \cdot \Res_{v_1}\phi= \Res_{v_1}( D \cdot \phi)=0.
$$
Thus $\Res_{v_1}\phi$ lies in the $\Sym(V_0^*)$-torsion module of $\CR_{\Delta_0}$. As  $\Res_{v_1}\phi \in B_{\Delta_0}$, we have $\Res_{v_1}\phi=0$ by the induction hypothesis. Thus $v_1$ is not a pole of $\phi$. Since $v_1$ was any element of $\Delta$, $\phi$ has no poles, so $\phi=0$.

Next, let us prove $G_\Delta\cap N_\Delta= \{0\}$.
    Let $\phi_1, \ldots, \phi_s$ be a basis of $B_\Delta$ (over the base field). Then any fraction $\phi\in G_\Delta$ can be written as
   $$
   \phi=\sum_{j=1}^s D_j \cdot \phi_j
   $$
for some $D_j$'s in $\Sym(V^*)$.  Assume $\phi\in N_\Delta$.  We want to prove that $D_j=0$ for every $j$.   We can assume that  $\phi $ is homogeneous, so $D_j$'s are homogeneous with the same  degree $k>0$ (otherwise, $\phi\in B_\Delta\cap N_\Delta$, so $\phi=0$). Assume  $D_1\neq 0$.  There exists $v\in V$ such that the Lie bracket $[v,D_1]=v\circ D_1-D_1\circ v$ is not  zero.  The brackets  $[v,D_j]$ are elements of $\Sym(V^*)$ of degree $k-1$.
We have
$ v\phi\in N_\Delta$ and $\sum_{j\in J} D_j\cdot v \phi_j\in  N_\Delta$, hence
$$
\sum_{j\in J}[v,D_j] \cdot \phi_j\in G_\Delta\cap N_\Delta.
$$
Iterating, we obtain a linear relation $\sum_{j\in J} c_j \phi_j\in N_\Delta$ with  $c_1\neq 0$, which contradicts the fact that $B_\Delta\cap N_\Delta=0$.
We have proved
$G_\Delta\cap N_\Delta= \{0\}$, hence  (i). We have also proved that $(\phi_1, \ldots, \phi_s)$ is  a basis of $G_\Delta$ as $\Sym(V^*)$-module, hence (iii), hence also (ii).
\end{proof}

We will now prove that $\Sym(C(L)) \otimes G_{\Delta\cap L}$ is a direct summand by induction on $\codim L$. The case $\codim L=0$ is Lemma \ref{lemma:arrangementsI}, (i). So, let us assume that $\Sym(C(L)) \otimes G_{\Delta\cap L}$ is a direct summand for all subspaces $L$ (spanned by elements of $\Delta$) which  have codimension $\leq k$. Let $L_0, L_1,\ldots, L_s$ be  the list of subspaces (spanned by elements of $\Delta$) which  have codimension $ k+1$.
If $s=0$, then we must have $k+1=\dim V$, hence $L_0=\{0\}$, so we are done. So we assume that $s\geq 1$.

Assume that we have a family $f_L\in \Sym(C(L)) \otimes G_{\Delta\cap L}$ such that
$$
\sum_{L, \codim L \leq k+1}f_L=0.
$$
We are going to prove that $f_{L_0}=0$.  For $j=1,\ldots, s$,  we choose $\gamma_j\in L_j^\perp\subset V^*$ such that $\gamma_j\notin L_0^\perp$.
For $N$ large enough, $D=\prod_{j=1}^s \partial_{\gamma_j}^N$
kills all terms $f_L$ if $L$ is contained in some $L_j$ with $j=1,\ldots,s$. So $D \cdot f_{L_0}=0$ as well.
We decompose
$$
V^*= C(L_0)^\perp \oplus L_0^\perp.
$$
Then
$$
\Sym(V^*)= \Sym(C(L_0))^\perp)\otimes \Sym(L_0^\perp)=\oplus_{n}S_n(C(L_0)^\perp)\otimes \Sym(L_0^\perp).
$$
Let
$D=D_0+D_{> 0}$ with respect  to the degree on $\Sym(C(L_0))^\perp)$. Let us show that $D_0\neq 0$.
Let
$\gamma_j= \alpha_j + \beta_j$ in the decomposition $V^*= C(L_0)^\perp \oplus L_0^\perp$.
Then $\alpha_j\neq 0$, and  $D_0= \prod_{j=1}^s \alpha_j^N$  is not zero.

We decompose $f_{L_0}\in \Sym(C(L_0)))\otimes G_{L_0\cap \Delta}= \oplus_{n}S_n(C(L_0))\otimes G_{L_0\cap \Delta}$ with respect to the degree on $\Sym(C(L_0))$,
$$
f_{L_0}=f_n+f_{n-1}+\cdots.
$$
We observe that $D_{> 0}$ lowers strictly the degree with respect to  $\Sym(C(L_0))$, while $D_0$ acts only on the  $G_{L_0\cap \Delta} $ factor. Therefore we obtain
$$
D_0\cdot f_n=0.
$$
We deduce $f_n=0$ by applying Lemma \ref{lemma:arrangementsI}, (iii),  with $V$ replaced by  the subspace $L_0$, and $L_0^*$ identified with  $C(L_0)^\perp $. Thus, we have proved $f_{L_0}=0$.
\end{proof}

\subsection{Renormalization of meromorphic functions with hyperplane singularities}\label{subsection:renorm-meromorphic}
From now on, we will assume that the complement map is given by a scalar product $Q$ on $V$. We will denote it by $C_Q$.
\begin{definition}\label{def:renormQ}
The projection map on  the summand corresponding to $L=\{0\}$ in Theorem \ref{th:structure} is called the renormalization map with respect to $Q$ and is denoted by
$$
R_Q: \CR_\Delta\to \Sym(V)
$$
\end{definition}
\begin{remark}\label{remark:renormalization}

The renormalization map extends into a map from the space of meromorphic functions with hyperplane singularities to the space of holomorphic functions near $\xi=0$.
$$
R_Q: \CM_h(V^*) \to \CH(V^*)
$$
First, if $f$ is a  rational fraction with hyperplane singularities, we apply Theorem \ref{th:structure} with $\Delta$ any set containing the  singular hyperplanes. For a meromorphic function $f$, we renormalize the Taylor expansion term by term.
\end{remark}
\begin{example}\label{ex:renormalization} In dimension one, if $\phi(\xi)$ is holomorphic near $0$, we have $\frac{\phi(\xi)}{\xi}=\frac{\phi(0)}{\xi}+\frac{\phi(\xi)-\phi(0)}{\xi}$, hence
$R_Q(\frac{\phi(\xi)}{\xi})= \frac{\phi(\xi)-\phi(0)}{\xi}$.
For instance,
$$
R_Q(\frac{1}{1-\e^\xi})=\frac{1}{1-\e^\xi} + \frac{1}{\xi}.
$$

In dimension 2, when $Q$ is the standard scalar product,  the decomposition of Theorem \ref{th:structure} is
\begin{multline*}
\frac{\phi(\xi_1,\xi_2)}{\xi_1\xi_2}= \frac{\phi(0,0)}{\xi_1\xi_2} +
\frac{1}{\xi_2}(\frac{\phi(\xi_1,0)-\phi(0,0)}{\xi_1})+
\frac{1}{\xi_1}(\frac{\phi(0,\xi_2)-\phi(0,0)}{\xi_2})\\ +\frac{\phi(\xi_1,\xi_2)-\phi(\xi_1,0)-\phi(0,\xi_2)+\phi(0,0)}{\xi_1\xi_2},
\end{multline*}
hence
 $$
 R_Q(\frac{\phi(\xi_1,\xi_2)}{\xi_1\xi_2})=\frac{\phi(\xi_1,\xi_2)-\phi(\xi_1,0)-\phi(0,\xi_2)+\phi(0,0)}{\xi_1\xi_2}.
 $$
\end{example}
For a function which has only simple poles which are linearly independent, renormalization provides a nice recursive formula for computing all the terms in the decomposition of Theorem \ref{th:structure}. It is implemented in a Maple program \cite{RenormalizationMaple}.
\begin{proposition}\label{th:structure-simple-poles}
 Let $\Delta=\{v_1,\ldots, v_s\}$ be a set of linearly independent vectors in $V$. For $J\subseteq \Delta$, let $L_J$ be the subspace of V spanned by $(v_j,j\in J)$.
Let $P\in \Sym(V)$ be a polynomial. Then the decomposition of $\frac{P}{v_1 \cdots v_s}$ in Theorem \ref{th:structure} is given by the following formula
\begin{equation}\label{eq:structure-simple-poles1}
  \frac{P}{v_1 \cdots v_s}=\sum_{J\subseteq \Delta } R_Q \left(\frac{P}{\prod_{k\notin J}v_k}|_{L_J^{\perp}}\right) \, \frac{1}{\prod_{j\in J}v_j}.
\end{equation}
Here, $L_J^\perp\subseteq V^*$ is identified with $C_Q(L_J)^*$.
\end{proposition}
Example \ref{ex:renormalization} illustrates the proposition.
\begin{proof}
The proof is by induction on $\dim V$.
The decomposition of $\phi=\frac{P}{v_1 \cdots v_s}$ takes the form
\begin{equation}\label{eq:structure-simple-poles2}
\frac{P}{v_1 \cdots v_s}=\sum_{J\subseteq \Delta }\frac{P_J}{\prod_{j\in J}v_j}.
\end{equation}
with $P_J\in \Sym(C_Q(L_J))$.

 Let us prove  that  $P_J= R_Q \left(\frac{P}{\prod_{k\notin J}v_k}|_{L_J^{\perp}}\right)$ if $1\in J$. We compute the residue  $\Res_{v_1}$ of both sides.
We obtain
\begin{equation}\label{eq:structure-simple-poles3}
\frac{\bar{P}}{\bar{v}_2\cdots \bar{v}_s}= \sum_{\{J, 1\in J \}}\frac{\bar{P}_J}{\prod_{j\in J, j>1}\bar{v}_j}
\end{equation}
where $\bar{P}, \bar{P}_J, \bar{v}_j$ denote the restriction to $v_1^\perp$. When we identify  $v_1^\perp$ with  $C_Q(v_1)^*$,  we have $ \bar{v}_j\in C_Q(v_1)$, $\bar{P}\in \Sym(C_Q(v_1))$, and (\ref{eq:structure-simple-poles3}) is the decomposition of $\frac{\bar{P}}{\bar{v}_2\cdots \bar{v}_s}$.  For $J$ such that  $1\in J$, we have $L_J^\perp\subseteq v_1^\perp$, so $\bar{P}_J=P_J$.  So, by induction,
$$
P_J= \bar{P}_J= R_Q\left(\frac{\bar{P}}{\prod_{k\notin J}\bar{v}_k}|_{\bar{L}_{\bar{J}}^\perp}\right),
$$
where $\bar{J}=J\setminus \{1\}$, $\bar{L}_{\bar{J}}\subset C_Q(v_1)$ is the space spanned by $(\bar{v}_j, j\in \bar{J})$ and ${\bar{L}_{\bar{J}}^\perp}$ denotes the orthogonal of $\bar{L}_{\bar{J}}$ in $C_Q(v_1)$. Thus, $\frac{\bar{P}}{\prod_{k\notin J}\bar{v}_k}|_{\bar{L}_{\bar{J}}^\perp}$ is identified with $
\frac{P}{\prod_{k\notin J}v_k}|_{L_J^{\perp}}$.
\end{proof}

\subsection{Renormalization of the generating function of a cone}\label{subsection:mu-by-renormalization}
In this section, $V$ is a rational space. We will give an alternative definition of the  $\mu$-function of an affine (rational  polyhedral) cone  as the renormalization of the generating function of the cone, and derive the properties of the $\mu$-function from this construction. Renormalization leads also a fast algorithm for computing the $\mu$-function for a simplicial cone in fixed dimension, using Barvinok's decomposition into unimodular cones.

\begin{definition}\label{def:mu-by-renormalization}
Fix a  scalar product  $Q$ on $V$.
\newline \noindent 1)
Let $\c\subset V$ be a cone and let $s\in V$.  Consider the affine cone $s+\c$. Define
$$
\mu_Q(s+\c)(\xi)=R_Q(\e^{-\langle \xi,s\rangle}S(s+\c)(\xi)).
$$
Then $\mu_Q(s+\c)(\xi) $ is a holomorphic function on $V^*$ near $\xi=0$
\newline \noindent 2)
Let $W$ be a rational quotient of $V$ and let $\a\subset W$ be an affine cone.  Then $W^*$ is a subspace of $V^*$.
The function $\mu_Q(\a)$ on $W^*$ is extended to $V^*$ by $0$ on the orthogonal complement of $W^*$ with respect to $Q$.
\end{definition}
The valuation  property  of the $\mu$-function follows immediately from this definition.

\begin{proposition}\label{prop:properties-of-mu}
 $\mu_Q$ is a valuation on the set of affine cones in $V$ with a fixed vertex $s$.
\end{proposition}
\begin{proof}
The map $\c\mapsto \e^{-\langle \xi,s\rangle}S(s+\c)(\xi)$ is a valuation on this set.
\end{proof}

We will now show that the local Euler-Maclaurin formula for generating functions of cones (\cite{MR2343137}, Theorem 20) can be easily derived from  this definition and the results on the poles and residues of $S(\a)$. (This result is also obtained in \cite{GuoPaychaZhang2015}).
\begin{theorem}\label{th:localEML}
    Let $V$ be a rational vector space. Fix a rational Euclidean scalar product on $Q$. Let $\a\subset V$ be a rational affine cone. If $\f$ is a face of $\a$, let $I(\f)$ be the continuous generating function  of $\f$ (cf. Section \ref{subsection:generating-function}) and let $\t(\a,\f)\subset V/\lin (\f)$ denote the transversal cone to $\a$ along $\f$. Then
    \begin{equation}\label{eq:localEML}
     S(\a)=\sum_\f \mu_Q(\t(\a,\f)) I(\f).
    \end{equation}
    The sum runs over the set of faces of $\a$.
\end{theorem}
\begin{proof}
The proof is by induction on $\dim \a$ and relies on the determination  of the residues of $S(\a)$ and $I(\a)$.
\begin{lemma}\label{lemma:residue}
The poles of $S(\a)$ and $I(\a)$ are the edge generators of $\a$.  The residues  $\Res_{v}S(\a)=vS(\a)|_{v^\perp}$ and
   $\Res_{v}I(\a)=vI(\a)|_{v^\perp}$ along an edge $v$ are given by
\begin{eqnarray}
  \Res_{v}S(\a)= - S(\proj_{V/v}\a),\label{eq:residueS}\\
  \Res_{v}I(\a)= - I(\proj_{V/v}\a),\label{eq:residueI}
\end{eqnarray}
    provided $v$ is a primitive lattice vector.
\end{lemma}
The analogous formula for any \emph{intermediate} generating function of $\a$ is  proven in (\cite{SLII2014}, Proposition 3.5), using a Poisson summation formula.
The following proof for the purely discrete and purely continuous generating functions is more elementary.
\begin{proof}
Using a subdivision of $\a$ into simplicial cones without adding edges, we can assume that $\a$ is simplicial. Furthermore,  we can also assume that $\a$ is full dimensional. Let $(v_1,\ldots ,v_d)$ be the primitive lattice edge generators of $\a$, with $v=v_1$. Let $s$ be the vertex of $\a$. Let $\b=\sum_{j}[0,1[v_j$ and $S(s+\b)(\xi)=\sum_{x\in (s+\b)\cap \lattice} \e^{\langle \xi, x\rangle}$, a finite sum,  so that $S(s+\b)(\xi)$ is analytic. We have
$$
S(\a)=S(s+\b)\frac{1}{\prod_{j=1}^d 1-\e^{v_j}}.
$$
We have $\Res_{v_1}\frac{1}{1-\e^{v_1}}=-1$, hence
$$
 \Res_{v_1}S(\a)= -S(\proj_{V/v_1}(s+\b))\, \frac{1}{\prod_{j=2}^d 1-\e^{\bar{v}_j}}.
$$
where $\bar{v}_j$ is the projection of $v_j$ on $V/v_1$.  The right-hand side is precisely $- S(\proj_{V/v_1}\a)$, even though the projected vectors $\bar{v}_j$ may not be primitive with respect to the projected lattice $\bar{\lattice}$. So we have proved (\ref{eq:residueS}).

We have $I(\a)=   \frac{(-1)^d |\det(v_1,\ldots,v_d)|\e^s}{\prod_{j=1}^d {v_j}}$, hence
$$ \Res_{v_1}I(\a)=
 \frac{(-1)^{d} |\det_\lattice(v_1,\ldots,v_d)|\e^{\bar{s}}}{\prod_{j=2}^d {\bar{v}_j}}.$$
Since $v_1$ is primitive, we have   $|\det_\lattice(v_1,\ldots,v_d)|=|\det_{\bar{\lattice}}(\bar{v}_2,\ldots,\bar{v}_d)|$. Therefore the right-hand side of the residue above is equal to $- I(\proj_{V/v_1}\a)$. So we have proved (\ref{eq:residueI}).
\end{proof}
Let $\Delta$ be the set of primitive edge generators of $\a$. We write the decomposition given by Theorem \ref{th:structure}.
\begin{equation}\label{eq:decompositionSa}
  S(\a)=\sum_{L} g_L.
\end{equation}
First, we observe that $ g_L\in \Sym(L^{\perp_Q})\otimes B_L$, in other words, $g_L$ has only simple poles.  This is clear if $\a$ is simplicial, so it is true also in the general case, using a subdivision of $\a$ with no edges added.

Pick $v\in \Delta$.  The set of faces of the projected cone $\proj_{V/v} \a$ is in one-to-one correspondence with the set of faces of $\a$ which have $v$ as an edge. For such a face $\f$, the transversal cone of  $\proj_{V/v} \a$ along its face $\proj_{V/v} \f$ coincides with the transversal cone $\t(\a,\f)\subset V/\lin(\f)$. Moreover, if $v$ is an edge of $\f$, we have
$I(\proj_{V/v} \f)=-\Res_{v}I(\f)$ by Lemma \ref{lemma:residue}.
On the other hand, if $v$ is not an edge of $\f$, we have $\Res_{v}I(\f)=0$.
So, by Lemma \ref{lemma:residue} and the induction hypothesis applied to the cone $\proj_{V/v} \a$, we have
 $$
  \Res_{v} S(\a)=\sum_{\f, v\in \f} \mu(\t(\a,\f)) I(\proj_{V/v} \f)=  \sum_\f \Res_{v}(\mu(\t(\a,\f)) I(\f)).
 $$
In conclusion,  we have $\Res_v g_L=0$ if $L$ is not a face and $\Res_v g_{\lin(\f)}= \Res_{v}(\mu(\t(\a,\f)) I(\f))$ for any edge $v$ of $\f$. It is clear that an element of $ \Sym(L^{\perp_Q})\otimes B_L$ is zero if all its residues are zero. Therefore, $g_L=0$ if $L$ is not a face and  $g_{\lin(\f)}= \mu(\t(\a,\f)) I(\f)$ for any face $\f$ of $\a$.
\end{proof}

\bibliographystyle{amsabbrv}
\bibliography{biblioEMLasymptotique}
\end{document}